\documentclass[graybox,envcountsame,envcountsect]{svmult}

\usepackage{amsmath,amssymb,latexsym}
\usepackage{enumerate}

\usepackage{bbm}

\usepackage{mathptmx}       
\usepackage{helvet}         
\usepackage{courier}        
\usepackage{type1cm}        
\usepackage{makeidx}         
\usepackage{graphicx}        
\usepackage{multicol}        
\usepackage[bottom]{footmisc}

\usepackage{caption}

\spnewtheorem*{sketch alternative}{Outline of an alternative proof}{\it}{}
\spnewtheorem*{alternative}{Alternative proof}{\it}{}
\spnewtheorem*{sketch}{Outline of the proof}{\it}{}
\spnewtheorem*{hint}{Hint}{\it}{}
\spnewtheorem*{notations}{Notation}{\bf}{}

\spnewtheorem{open problem}[theorem]{Open problem}{\bf}{}

\smartqed

\newcommand{\RR}{\mathbb{R}}

\newcommand{\pa}{\partial}

\newcommand{\ve}{\varepsilon}

\newcommand{\ol}{\overline}

\newcommand{\la}{\lambda}
    
\newcommand{\de}{\delta}    
\newcommand{\fhi}{\varphi} 
\newcommand{\ga}{\gamma}

\newcommand{\dv}{\mathop{\mathrm{div}}}

\newcommand{\al}{\alpha}
\newcommand{\be}{\beta}
\newcommand{\Om}{\Omega}
\newcommand{\na}{\nabla}

\newcommand{\nr}{\Vert}
\newcommand{\De}{\Delta}

\newcommand{\si}{\sigma}
\newcommand{\te}{\theta}
\newcommand{\Ga}{\Gamma}

\newcommand{\cSC}{{\mathcal C}_S}
\newcommand{\cH}{{\mathcal H}}

\newcommand{\DLB}{\De_{LB} \, }

\makeindex             

\begin{document}

\title*{Stable solutions to some elliptic problems: minimal cones, the Allen-Cahn equation, and blow-up solutions}
\titlerunning{Stable solutions to some elliptic problems}
\author{Xavier Cabr\'e and Giorgio Poggesi}
\institute{Xavier Cabr\'e\textsuperscript{1,2} \at \textsuperscript{1} Universitat Polit\`ecnica de Catalunya, Departament de Matem\`{a}tiques, Diagonal 647, 
08028 Barcelona, Spain
\at \textsuperscript{2} ICREA, Pg. Lluis Companys 23, 08010 Barcelona, Spain 
\\
\email{xavier.cabre@upc.edu}
\and
Giorgio Poggesi \at Dipartimento di Matematica ed Informatica ``U.~Dini'',
Universit\` a di Firenze, viale Morgagni 67/A, 50134 Firenze, Italy
\\
\email{giorgio.poggesi@unifi.it}
}
%
%


\maketitle

\abstract*{These notes record the lectures for the CIME Summer Course taught by the first author in Cetraro during the week of June 19-23, 2017.
The notes contain the proofs of several results on the classification of stable solutions to some nonlinear elliptic equations. The results are crucial steps within the regularity theory of minimizers to such problems. We focus our attention on three different equations, emphasizing that the techniques and ideas in the three settings are quite similar.
\newline\indent
The first topic is the stability of minimal cones. We prove the minimality of the Simons cone in high dimensions, and we give almost all details in the proof of J.~Simons on the flatness of stable minimal cones in low dimensions.  
\newline\indent
Its semilinear analogue is a conjecture on the Allen-Cahn equation posed by E.~De Giorgi in 1978. This is our second problem, for which we discuss some results, as well as an open problem in high dimensions on the saddle-shaped solution vanishing on the Simons cone.
\newline\indent
The third problem was raised by H.~Brezis around 1996 and concerns the boundedness of stable solutions to reaction-diffusion equations in bounded domains. We present proofs on their regularity in low dimensions and discuss the main open problem in this topic. 
\newline\indent
Moreover, we briefly comment on related results for harmonic maps, free boundary problems, and nonlocal minimal surfaces.}

\abstract{These notes record the lectures for the CIME Summer Course taught by the first author in Cetraro during the week of June 19-23, 2017.
The notes contain the proofs of several results on the classification of stable solutions to some nonlinear elliptic equations. The results are crucial steps within the regularity theory of minimizers to such problems. We focus our attention on three different equations, emphasizing that the techniques and ideas in the three settings are quite similar.
\newline\indent
The first topic is the stability of minimal cones. We prove the minimality of the Simons cone in high dimensions, and we give almost all details in the proof of J.~Simons on the flatness of stable minimal cones in low dimensions.  
\newline\indent
Its semilinear analogue is a conjecture on the Allen-Cahn equation posed by E.~De Giorgi in 1978. This is our second problem, for which we discuss some results, as well as an open problem in high dimensions on the saddle-shaped solution vanishing on the Simons cone.
\newline\indent
The third problem was raised by H.~Brezis around 1996 and concerns the boundedness of stable solutions to reaction-diffusion equations in bounded domains. We present proofs on their regularity in low dimensions and discuss the main open problem in this topic. 
\newline\indent
Moreover, we briefly comment on related results for harmonic maps, free boundary problems, and nonlocal minimal surfaces.}



\newpage

\setcounter{minitocdepth}{2}
\dominitoc

The abstract and table of contents above give an account of the topics treated in these lecture notes.

\section{Minimal cones}
\label{sec:mincones}
In this section we discuss two classical results on the theory of minimal surfaces: Simons flatness result on stable minimal cones in low dimensions and the Bombieri-De Giorgi-Giusti counterexample in high dimensions.
The main purpose of these lecture notes is to present the main ideas and computations leading to these deep results -- and to related ones in subsequent sections. Therefore, to save time for this purpose, we do not consider the most general classes of sets or functions (defined through weak notions), but instead we assume them to be regular enough.

Throughout the notes, for certain results we will refer to three other expositions: the books of Giusti \cite{G}
and of Colding and Minicozzi \cite{CM}, and the CIME lecture notes of Cozzi and Figalli \cite{CF}.
The notes \cite{CabCapThree} by the first author and Capella have a similar spirit to the current ones and may complement them.

\begin{definition}[Perimeter]
Let $E \subset \RR^n$ be an open set, regular enough. For a given open ball $B_R$ we define the {\it perimeter of} $E$ {\it in} $B_R$ as 
\begin{equation*}
P(E, B_R) := H_{n-1} (\pa E \cap B_R ),
\end{equation*}
where $H_{n-1}$ denotes the $(n-1)$-dimensional Hausdorff measure (see Figure~\ref{fig:1}).
\end{definition}

The interested reader can learn from \cite{G, CF} a more general notion of perimeter (defined by duality or in a weak sense) and the concept of set of finite perimeter.

\begin{definition}[Minimal set]\label{def:minimal set}
We say that an open set (regular enough) $E \subset \RR^n$ is a {\it minimal set} (or a {\it set of minimal perimeter}) if and only if, for every given open ball $B_R$, it holds that
$$P(E, B_R)\leq P(F, B_R)$$
for every open set $F\subset\RR^n$ (regular enough) such that $E \setminus B_R = F\setminus B_R$.
\end{definition}
In other words, $E$ has least perimeter in $B_R$ among all (regular) sets which agree with $E$ outside $B_R$.

\begin{figure}[htbp]
\centering
\includegraphics[scale=.25]{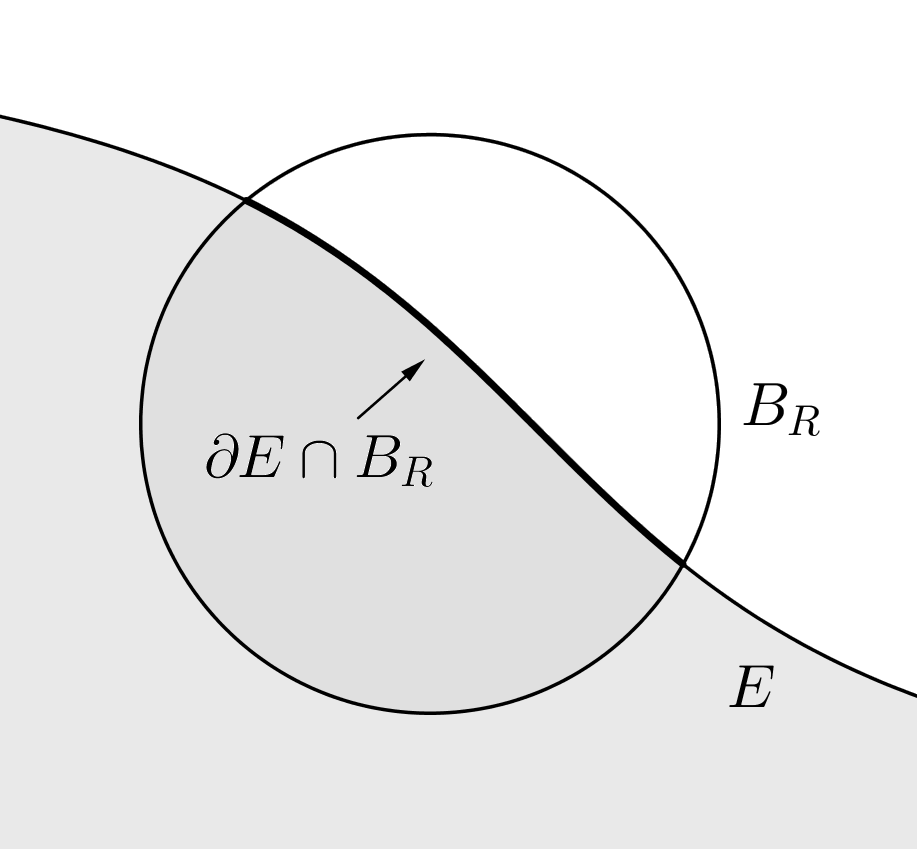}
%
%
\caption{The perimeter of $E$ in $B_R$}
\label{fig:1}       
\end{figure}

To proceed, one considers small perturbations of a given set $E$
and computes the {\it first and second variations of the perimeter functional}. 
To this end,
let $\{\phi_t\}$ be a one-parameter family of maps $\phi_t:\RR^n \to \RR^n$
such that $\phi_0$ is the identity $I$ and all the maps $\phi_t-I$ have compact support (uniformly) contained in $B_R$.

Consider the sets $E_t=\phi_t(E)$. We are interested in the perimeter functional $P(E_t, B_R)$.
One proceeds by choosing $\phi_t=I+t\xi\nu$, which shifts the original set $E$ in the normal direction $\nu$ to its boundary. Here $\nu$ is the outer normal to $E$ and is extended in a neighborhood of $\pa E$ to agree with the gradient of the signed distance function to $\pa E$, as in \cite{G} or in our Subsection \ref{subsec:SimLem} below. On the other hand, $\xi$ is a scalar function with compact support in $B_R$ (see Figure \ref{fig:2}).

\begin{figure}[htbp]
\centering
\includegraphics[scale=.25]{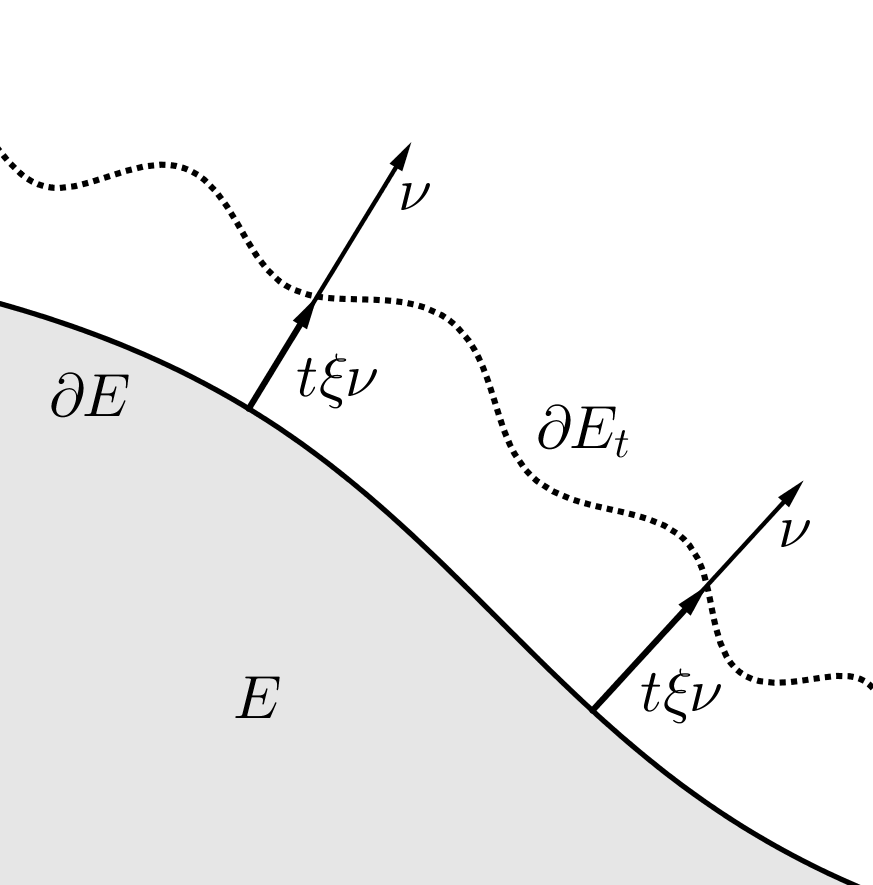}
\caption{A normal deformation $E_t$ of $E$}
\label{fig:2}       
\end{figure}

It can be proved (see chapter 10 of \cite{G}) that the first and second variations of perimeter are given by
\begin{eqnarray}
\left.\frac{d}{dt} P(E_t, B_R) \right|_{t=0} 
&=&\int_{\partial E}{\mathcal H}\xi dH_{n-1},
\label{eq:1-1v}\\
\left.\frac{d^2}{dt^2}  P(E_t, B_R) \right|_{t=0}
&=&\int_{\partial E}\left\{|\delta\xi|^2-(c^2-{\mathcal H}^2)\xi^2
\right\}dH_{n-1},
\label{eq:1-1vBIS}
\end{eqnarray}
where ${\mathcal H}={\mathcal H}(x)$ is the {\it mean curvature} of $\partial E$ at $x$ and $c^2=c^2(x)$ is the sum of the squares 
of the $n-1$ principal
curvatures $k_1, \dots, k_{n-1}$ of $\pa E$ at $x$. More precisely,
$$
\cH(x)= k_1 + \dots + k_{n-1} \quad \mbox{ and } \quad c^2= k_1^2 + \dots + k_{n-1}^2 .
$$
In \eqref{eq:1-1vBIS}, $\delta$ (sometimes denoted by $\na_T$) is the tangential gradient to the surface
$\pa E$, given by
\begin{equation}\label{def:tangentialgradient}
\delta \xi = \na_T \xi = \na \xi - (\na \xi \cdot \nu) \nu
\end{equation}
for any function $\xi$ defined in a neighborhood of $\pa E$. Here $\na$ is the usual Euclidean gradient and $\nu$ is always the normal vector to $\pa E$.
Being $\de$ the tangential gradient, one can check that $\de \xi_{| \pa E}$ depends only on $\xi_{| \pa E}$. It can be therefore computed for functions $\xi: \pa E \to \RR$ defined only on $\pa E$ (and not necessarily in a neighborhood of $\pa E$).


\begin{definition}
\begin{enumerate}[(i)]
\item We say that $\pa E$ is a {\it minimal surface} (or a {\it stationary surface}) if the first variation of perimeter vanishes for all balls $B_R$. Equivalently, by \eqref{eq:1-1v}, $\cH = 0$ on $\pa E$.
\item We say that $\pa E$ is a {\it stable minimal surface} if $\cH =0$ and the second variation of perimeter is nonnegative for all balls $B_R$.
\item We say that $\pa E$ is a {\it minimizing minimal surface} if $E$ is a minimal set as in Definition~\ref{def:minimal set}.
\end{enumerate}
\end{definition}
We warm the reader that in some books or articles ``minimal surface'' may mean ``minimizing minimal surface''.

\begin{remark}
\begin{enumerate}[(i)]
\item If $\pa E$ is a minimal surface (i.e., $\cH = 0$), the second variation
of perimeter \eqref{eq:1-1vBIS} becomes
\begin{equation}
\left.\frac{d^2}{dt^2}  P(E_t, B_R) \right|_{t=0} = \int_{\partial E}\left\{|\delta\xi|^2-c^2\xi^2\right\}dH_{n-1} .
\label{eq:minimal-1-2v2}
\end{equation}
\item If $\pa E$ is a minimizing minimal surface, then $\pa E$ is a stable minimal surface. In fact, in this case the function $P(E_t, B_R)$
has a global minimum at $t=0$.
\end{enumerate}
\end{remark}

\subsection{The Simons cone. Minimality}\label{subsec 1.1:MinimalitySimcones}

\begin{definition}[The Simons cone]
The Simons cone $\cSC \subset \RR^{2m}$ is the set
\begin{equation}\label{def:simonscone}
\cSC =\{ x\in\mathbb{R}^{2m}\, : \, x_1^2+\ldots +x_m^2 = x_{m+1}^2
+\ldots +x_{2m}^2 \} .
\end{equation}
In what follows we will also use the following notation:
\begin{equation*}
\cSC =\{ x=(x',x'') \in \RR^m \times \RR^m \, : |x'|^2=|x''|^2 \}.
\end{equation*}
Let us consider the open set
\begin{equation*}
E_S= \left\{x \in \RR^{2m} :  u(x):= |x'|^2 - |x''|^2 < 0 \right\},
\end{equation*}
and notice that $\pa E_S = \cSC$ (see Figure \ref{fig:3}).
\end{definition}

\begin{figure}[htbp]
\centering
\includegraphics[scale=.25]{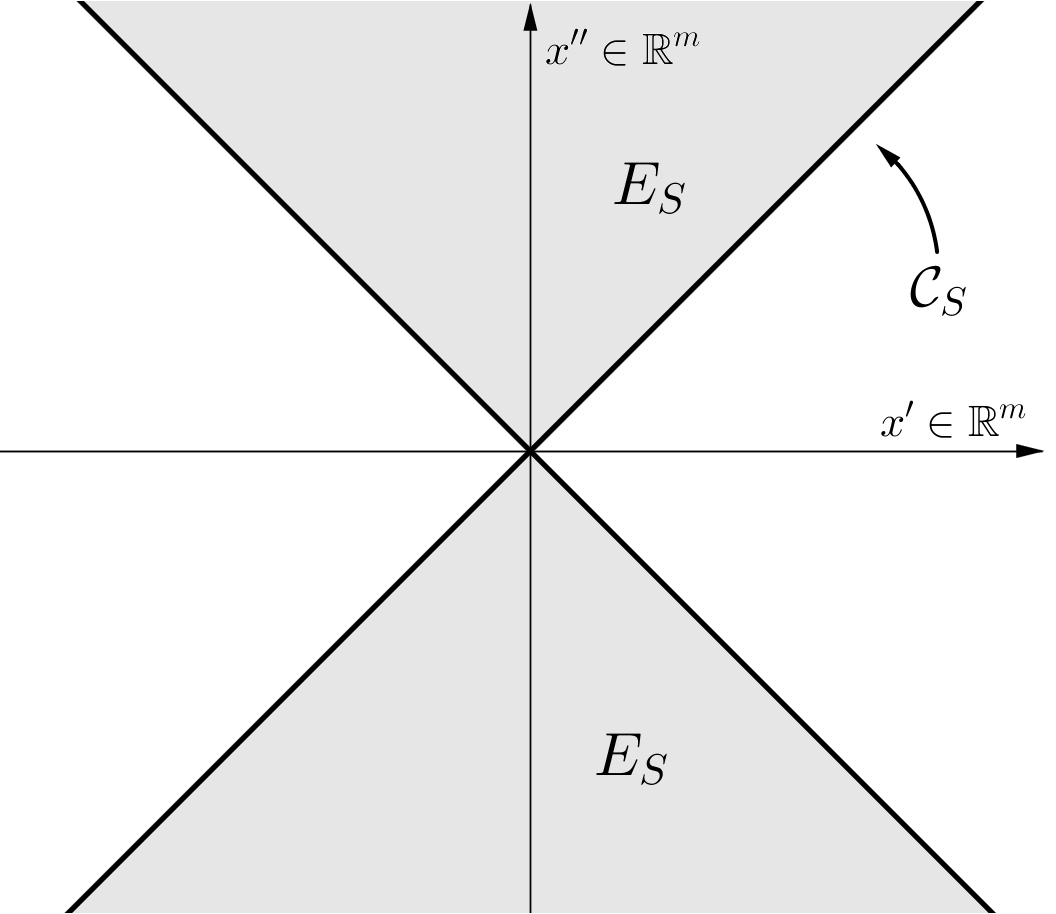}
\caption{The set $E_S$ and the Simons cone $\cSC$}
\label{fig:3}       
\end{figure}

\begin{exercise}
Prove that the Simons cone has zero mean curvature for every integer $m \ge 1$. For this, use the following fact (that you may also try to prove): if
$$
E=\left\{x \in \RR^n :  u(x) < 0 \right\}
$$
for some function $u: \RR^n \to \RR$, then the mean curvature of $\pa E$ is given by
\begin{equation}
\label{mean curvature formula}
{\mathcal H} = \left. \dv \left( \frac{\na u}{| \na u|} \right) \right|_{\pa E}.
\end{equation}
\end{exercise}

\begin{remark}\label{remark:nominimindim2}
It is easy to check that, in $\RR^2$, $\cSC$ is not a minimizing minimal surface. In fact, referring to Figure \ref{fig:4}, the shortest way to go from $P_1$ to $P_2$ is through the straight line. Thus, if we consider as a competitor in $B_R$ the interior of the set
$$F:= \ol{E}_S \cup \ol{T}_1 \cup \ol{T}_2,$$
where $T_1$ is the triangle with vertices $O$, $P_1$, $P_2$, and $T_2$ is the symmetric of $T_1$ with respect to $O$, we have that $F$ has less perimeter in $B_R$ than $E_S$.

\begin{figure}[htbp]
\centering
\includegraphics[scale=.25]{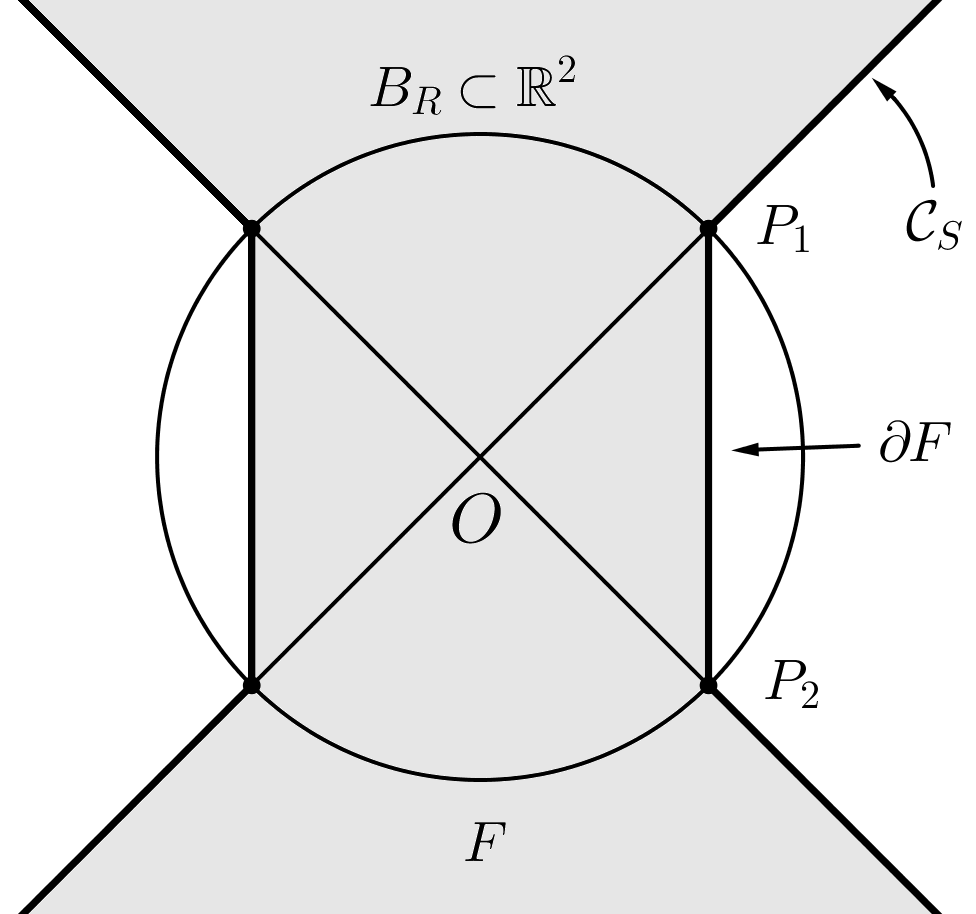}
\caption{The Simons cone $\cSC$ is not a minimizer in $\RR^2$}
\label{fig:4}       
\end{figure}

\end{remark}

In 1969 Bombieri, De Giorgi, and Giusti proved the following result.

\begin{theorem}[Bombieri-De Giorgi-Giusti~\cite{BdGG}]
\label{thm:SimCone} 
If $2m\geq 8$, then $E_S$ is a minimal set in $\mathbb{R}^{2m}$. That is, if  $2m\geq 8$, the Simons cone $\cSC$ is a minimizing minimal surface.
\end{theorem}

The following is a clever proof of Theorem \ref{thm:SimCone} found in 2009 by G.~De Philippis and E.~Paolini (\cite{DP}). It is based on a {\it calibration argument}.
Let us first define
\begin{equation}\label{def:utilde1}
\tilde{u} = |x'|^4 - |x''|^4  ;
\end{equation}
clearly we have that
\begin{equation*}
E_S= \left\{x \in \RR^{2m} :  \tilde{u}(x) < 0 \right\} \, \mbox{ and } \, \pa E_S = \cSC.
\end{equation*}
Let us also consider the vector field
\begin{equation}\label{def:Xvectorfield}
X= \frac{\na \tilde{u}}{ |\na \tilde{u}| } .
\end{equation}

\begin{exercise}\label{ex:utilde}
Check that if $m \ge 4$, $\dv X$ has the same sign as $\tilde{u}$ in $\RR^{2m}$.
\end{exercise}

\begin{proof}[of Theorem  \ref{thm:SimCone}]
By Exercise \ref{ex:utilde} we know that if $m \ge 4$, $\dv X$ has the same sign as $\tilde{u}$, where $\tilde{u}$ and $X$ are defined in \eqref{def:utilde1} and \eqref{def:Xvectorfield}.
Let $F$ be a competitor for $E_S$ in a ball $B_R$, with $F$ regular enough. We have that
$F \setminus B_R = E_S \setminus B_R$.


Set $\Om := F \setminus E_S $ (see Figure \ref{fig:5}). By using the fact that $\dv X \geq 0$ in $\Om$ and the divergence theorem, we deduce that
\begin{equation}
\label{calideph}
0 \leq \int_{\Om} \dv X \, dx= \int_{\partial E_S \cap \ol{\Om} } X \cdot \nu_\Om \, dH_{n-1} + \int_{\partial F \cap \ol{\Om}} X \cdot \nu_\Om \, dH_{n-1}.
\end{equation}

\begin{figure}[htbp]
\centering
\includegraphics[scale=.25]{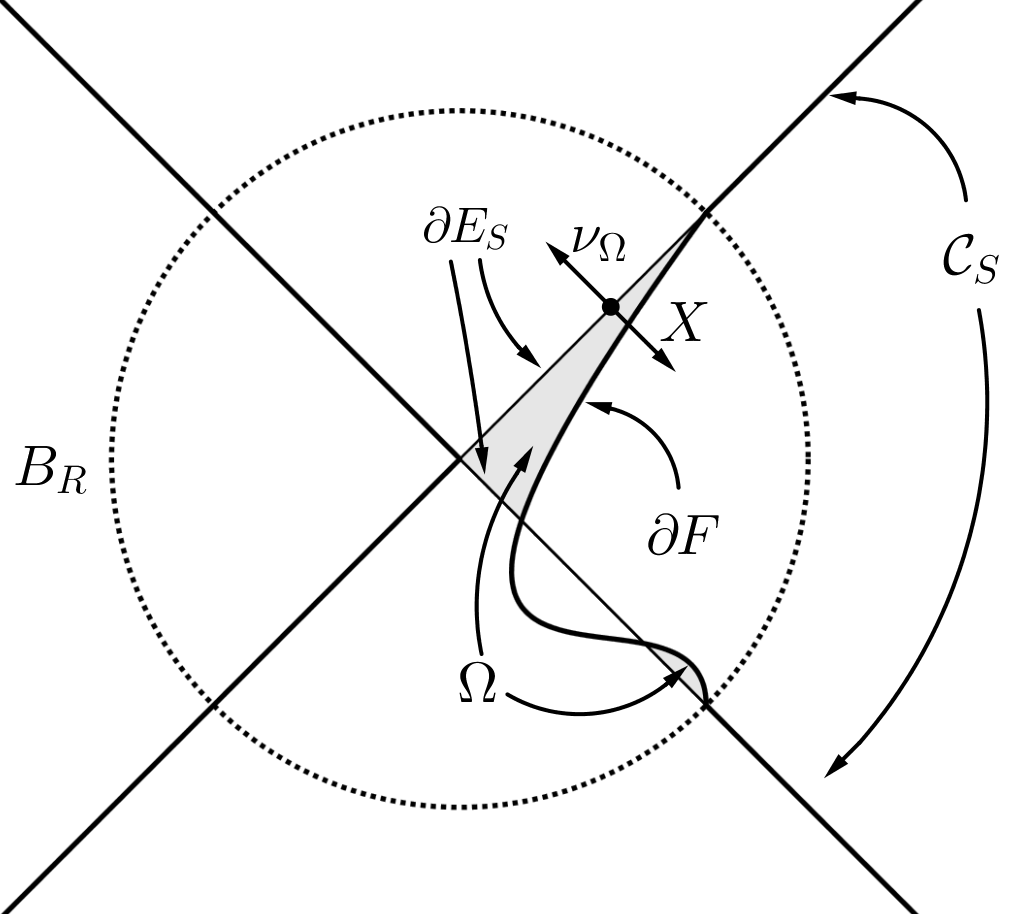}
\caption{A calibration proving that the Simons cone $\cSC$ is minimizing}
\label{fig:5}       
\end{figure}

Since $X= \nu_{E_S}= - \nu_\Om$ on $\partial E_S \cap \ol{\Om}$, and $|X| \leq 1$ (since in fact $|X|= 1$) everywhere (and hence in particular on $\partial F \cap \ol{\Om}$), from \eqref{calideph} we conclude 
\begin{equation}\label{eq:dimcalibration1}
H_{n-1} (\pa E_s \cap \ol{\Om}) \le H_{n-1} (\pa F \cap \ol{\Om}) .
\end{equation}

With the same reasoning it is easy to prove that \eqref{eq:dimcalibration1} holds also for
$\Om := E_S \setminus F$.
Putting both inequalities together, we conclude that $P(E_S,B_R) \leq P(F, B_R)$.

Notice that the proof works for competitors $F$ which are regular enough (since we applied the divergence theorem). However, it can be generalized to very general competitors by using the generalized definition of perimeter, as in \cite[Theorem~1.5]{DP}.
\qed
\end{proof}

Theorem \ref{thm:SimCone} can also be proved with another argument -- but still very much related to the previous one and that also uses the function $\tilde{u} = |x'|^4 - |x''|^4 $. It consists of going to one more dimension $\RR^{2m+1}$ and working with the minimal surface equation for graphs, \eqref{equa:Hgraph} below. This is done in Theorem 16.4 of \cite{G} (see also the proof of Theorem 2.2 in \cite{CabCapThree}).

In the proof above we used a vector field $X$ satisfying the following three properties (with $E= E_S$):
\begin{enumerate}[(i)]
\item $\dv X \geq 0$ in $B_R \setminus E$ and $\dv X \leq 0$ in $E \cap B_R$;
\item $X= \nu_{E}$ on $\partial E \cap B_R$;
\item $|X| \leq 1$ in $B_R$.
\end{enumerate}

\begin{definition}[Calibration]\label{def:calibrationprima}
If $X$ satisfies the three properties above we say that $X$  is a \textit{calibration} for $E$ in $B_R$.
\end{definition}

\begin{exercise}
Use a similar argument to that of our last proof and build a calibration to show that a hyperplane in $\RR^n$ is a minimizing minimal surface.
\end{exercise}

In an appendix, and with the purpose that the reader gets acquainted with another calibration, we present one which solves the isoperimetric problem: balls minimize perimeter among sets of given volume in $\RR^n$.
Note that the first variation (or Euler-Lagrange equation) for this problem is, by Lagrange multipliers,
$\cH = c$, where $c \in \RR$ is a constant.

The following is an alternative proof of Theorem \ref{thm:SimCone}. It uses a {\it foliation argument}, as explained below.
This second proof is probably more transparent (or intuitive) than the previous one and it is used often in minimal surfaces theory, but requires to know the existence of a (regular enough) minimizer (something that was not necessary in the previous proof). This existence result is available and can be proved with tools of the Calculus of Variations (see \cite{CF,G}).

The proof also requires the use of the following important fact. If $\Sigma_1$, $\Sigma_2 \subset B_R$ are two connected hypersurfaces (regular enough), both satisfying $\cH=0$, and such that $\Sigma_1 \cap \Sigma_2 \neq \varnothing$ and $\Sigma_1$ lies on one side of $\Sigma_2$, then $\Sigma_1 \equiv \Sigma_2$ in $B_R$. Lying on one side can be defined as $\Sigma_1 = \pa F_1$, $\Sigma_2 = \pa F_2$, and $F_1 \subset F_2$. The same result holds if $F_1$ satisfies $\cH=0$ and $F_2$ satisfies $\cH \geq 0$.

This result can be proved writing both surfaces as graphs in a neighborhood of a common point $P \in \Sigma_1 \cap \Sigma_2$. The minimal surface equation $\cH=0$ then becomes
\begin{equation}\label{equa:Hgraph}
\dv \left( \frac{\na \fhi_1}{ \sqrt{1+| \na \fhi_1|^2 }  } \right) =0
\end{equation}
for $\fhi_1: \Om \subset \RR^{n-1} \to \RR$ such that $\left( y', \fhi_1 (y') \right) \subset \Om \times \RR $ is a piece of $\Sigma_1$ (after a rotation and translation). Then, assuming that $\fhi_2$ also satisfies \eqref{equa:Hgraph} -- or the appropriate inequality --, one can see that $\fhi_1 - \fhi_2$ is a (super)solution of a second order linear elliptic equation. Since $\fhi_1 - \fhi_2 \geq 0$ (due to the ordering of $\Sigma_1$ and $\Sigma_2$), the strong maximum principle leads to $\fhi_1 - \fhi_2 \equiv 0$ (since $(\fhi_1 -\fhi_2)(0)=0$ at the touching point). See Section 7 of Chapter 1 of \cite{CM} for more details.

\begin{alternative}[of Theorem \ref{thm:SimCone}]
Note that the hypersurfaces
$$\left\{ x \in \RR^{2m} : \tilde{u} (x) = \la \right\} , $$
with $\la \in \RR$, form a foliation of $\RR^{2m}$, where $\tilde{u}$ is the function defined in \eqref{def:utilde1}.

Let $F$ be a minimizer of the perimeter in $B_R$ among sets that coincide with $E_S$ on $\pa B_R$, and assume that it is regular enough. Since $F$ is a minimizer, in particular $\pa F$ is a solution of the minimal surface equation $\cH=0$.
Since $2m \geq 8$, by \eqref{mean curvature formula} and Exercise \ref{ex:utilde}, the leaves of our foliation $\left\{ x \in \RR^{2m} : \tilde{u} (x) = \la \right\}$ are subsolutions of the same equation for $\la>0$, and supersolutions for $\la<0$ .

If $F \not\equiv E_S$, there will be a first leaf (starting either from $\la= + \infty$ or from $\la = - \infty$) $\left\{ x \in \RR^{2m} : \tilde{u} (x) = \la_* \right\}$, with $\la_* \neq 0$, that touches $\pa F$ at a point in $\ol{B}_R$ that we call $P$ (see Figure \ref{fig:6}).

\begin{figure}[htbp]
\centering
\includegraphics[scale=.25]{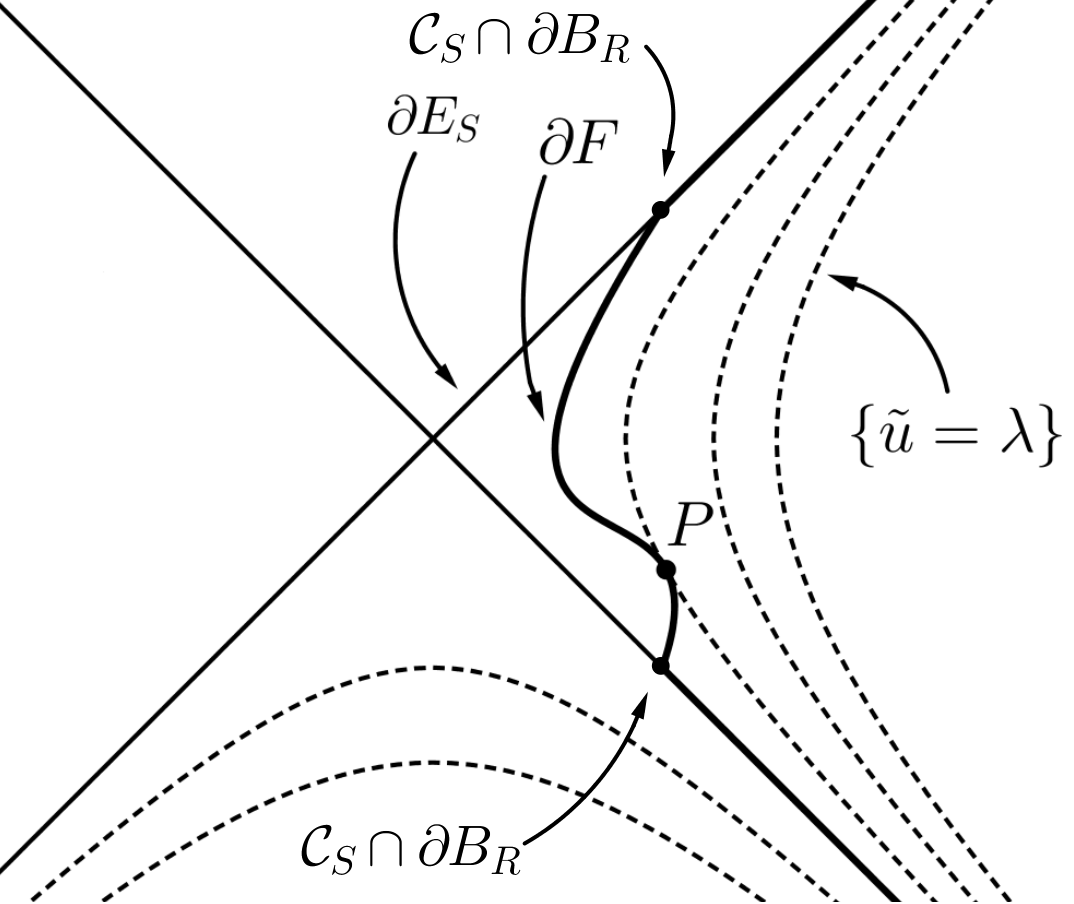}
\caption{The foliation argument to prove that the Simons cone $\cSC$ is minimizing}
\label{fig:6}       
\end{figure}

The point $P$ cannot belong to $\pa B_R$, since it holds that
$$\pa F \cap \pa B_R = \cSC \cap \pa B_R = \left\{ x: \tilde{u} (x) = 0 \right\} \cap \pa B_R  ,$$
and the level sets of $\tilde{u}$ do not intersect each other.
Thus, $P$ must be an interior point of $B_R$. But then we arrive at a contradiction, by the ``strong maximum principle'' argument commented right before this proof, applied with
$\Sigma_1= \pa F$ and $\Sigma_2= \left\{ x \in \RR^{2m} : \tilde{u} (x) = \la_* \right\} $.

As an exercise, write the details to prove the existence of a first leaf touching $\pa F$ at an interior point.

This same foliation argument will be used, in a simpler setting for graphs and the Allen-Cahn equation, in the proof of Theorem \ref{alba} in the next section.
\qed
\end{alternative}

\begin{remark}
The previous foliation argument gives more than the minimality of $\cSC$. It gives {\it uniqueness for the Dirichlet (or Plateau) problem} associated to the minimal surface equation with $\cSC$ as boundary value on $\pa B_R$. 
\end{remark}

\begin{remark}\label{rem:Foliation B-DG-G}
In our alternative proof of Theorem \ref{thm:SimCone} we used a clever foliation made of subsolutions and supersolutions. This sufficed to prove in a simple way Theorem \ref{thm:SimCone}, but required to (luckily) find the auxiliary function $\tilde{u} = |x'|^4 - |x''|^4 $.
Instead, in \cite{BdGG}, Bombieri, De Giorgi, and Giusti considered the foliation made of exact solutions to the minimal surface equation $\cH=0$, when $2 m \ge 8$. To this end, they proceeded as in the following exercise and wrote the minimal surface equation, for surfaces with rotational symmetry in $x'$ and in $x''$, as an ODE in $\RR^2$, finding Equation \eqref{equa:odest} below. They then showed that the solutions of such ODE in the $(s,t)$-plane do not intersect each other (and neither the Simons cone), and thus form a foliation (see Remark \ref{rem:dopohardy} for more information on this).
\end{remark}

\begin{exercise}
Let us set $s= |x'|$ and $t=|x''|$ for $x=(x', x'') \in \RR^m \times \RR^m$.
Check that the following two ODEs are equivalent to the minimal surface equation $\cH=0$ written in the $(s,t)$-variables for surfaces with rotational symmetry in $x'$ and in $x''$.
\begin{enumerate}[(i)]
\item As done in \cite{BdGG}, if we set a parametric representation $ s=s(\tau)$, $t = t(\tau)$, we find
\begin{equation}\label{equa:odest}
s''t'-s't''+(m-1) \left( (s')^2 + (t')^2 \right) \left( \frac{s'}{t} - \frac{t'}{s} \right) = 0 ;
\end{equation}
\item as done in \cite{Da}, if we set $s=e^{z(\te)} \cos (\te)$, $t = e^{z(\te)} \sin (\te)$ we get
$$
z''= \left( 1+(z')^2 \right) \left( (2m -1) - \frac{2(m-1) \cos (2 \te)}{ \sin (2 \te)} \, z' \right) .
$$
\end{enumerate}
The previous ODEs can be found starting from \eqref{mean curvature formula} when $u=u(s,t)$ depends only on $s$ and $t$. Alternatively, they can also be found computing the first variation of the perimeter functional in $\RR^{2m}$ written in the $(s,t)$-variables:
\begin{equation}\label{equa:stenergy}
c \int s^{m-1} t^{m-1} \, dH_1 (s,t),
\end{equation}
for some positive constant $c$, that becomes
%
$$
c \int e^{(2m-1)z(\te)} \, \cos^{m-1} (\te) \, \,  \sin^{m-1} (\te) \, \,  \sqrt{1+ \left( z'(\te) \right)^2 } \, d \te
$$
with the parametrization in point (ii).
\end{exercise}

\begin{remark}
For $n \ge 8$, there exist other minimizing cones, such as some of the {\it Lawson's cones}, defined by
$$
{\mathcal C}_L= \left\{ y= (y', y'') \in \RR^k \times \RR^{n-k} :  |y'|^2= c_{n,k} \, |y''|^2 \right\} \quad \text{for } k \ge 2 \text{ and } n-k \ge 2 .
$$
For details, see \cite{Da}.
\end{remark}

Notice that if $\pa E$ is a {\it cone} (i.e., $\la \pa E = \pa E$ for every $\la>0$), in the expressions \eqref{eq:1-1v}, \eqref{eq:1-1vBIS}, and \eqref{eq:minimal-1-2v2} we will always consider $\xi$ with compact support outside the origin (thus, not changing the possible singularity of the cone at the origin).

The next theorem was proved by Simons in 1968\footnote{Theorem \ref{thmcone} was proved in 1965 by De Giorgi for $n=3$, in 1966 by Almgren for $n=4$, and finally in 1968 by Simons in any dimension $n \le 7$.} (it is Theorem 10.10 in \cite{G}).
It is a crucial result towards the regularity theory of minimizing minimal surfaces.

\begin{theorem}[Simons~\cite{S}]\label{thmcone}
Let $E \subset \RR^n$ be an open set such that $\pa E$ is a stable minimal cone and 
$\partial E\setminus\{0\}$ is regular. 
Thus, we are assuming $\cH=0$ and
\begin{equation}
\int_{\partial E}\left\{|\delta\xi|^2-c^2\xi^2\right\}dH_{n-1}\geq 0
\label{eq:1-2v2}
\end{equation}
for every $\xi\in C^1(\partial E)$ with 
compact support outside the origin.

If $3 \leq n\leq 7$, then $\partial E$ is a hyperplane.
\end{theorem}

\begin{remark}\label{remarkchemiserve:1.17}
Simons result (Theorem~\ref{thmcone}), together with a blow-up argument and a monotonicity formula (as main tools), lead to the flatness of every minimizing minimal surface in all of $\RR^n$ if $n \le 7$ (see \cite[Theorem 17.3]{G} for a proof). The same tools also give the analyticity of every minimal surface that is minimizing in a given ball of $\RR^n$ if $n \le 7$ (see \cite[Theorem 10.11]{G} for a detailed proof). See also \cite{CF} for a great shorter exposition of these results.
\end{remark}

The dimension $7$ in Theorem \ref{thmcone} is optimal, since by Theorem \ref{thm:SimCone} the Simons cone provides a counterexample in dimension $8$.

The following is a very rough explanation of why the minimizer of the Dirichlet (or Plateau) problem is the Simons cone (and thus passes through the origin) in high dimensions -- in opposition with low dimensions, as in Figure \ref{fig:4}, where the minimizer stays away from the origin. In the perimeter functional written in the $(s,t)$-variables \eqref{equa:stenergy}, the Jacobian $s^{m-1} t^{m-1}$ becomes smaller and smaller near the origin as $m$ gets larger. Thus, lengths near $(s,t)=(0,0)$ become smaller as the dimension $m$ increases.

In order to prove Theorem \ref{thmcone}, we start with some important preliminaries. 
Recalling \eqref{def:tangentialgradient}, for $i=1 \dots ,n$, we define the tangential derivative
$$
\de_i \xi := \pa_i \xi - \nu^i \, \nu^k \xi_k ,
$$
where $\nu= \nu_E = (\nu^1, \dots, \nu^n): \pa E \to \RR^n$ is the exterior normal to $E$ on $\pa E$,
$\pa_i \xi = \pa_{x_i} \xi = \xi_i$ are Euclidean partial derivatives, and we used the standard convention of sum $\sum_{k=1}^{n}$ over repeated indices.
As mentioned right after definition \eqref{def:tangentialgradient}, even if to compute $\pa_i \xi$ requires to extend $\xi$ to a neighborhood of $\pa E$, $\de_i \xi$ is well defined knowing $\xi$ only on
$\pa E$ -- since it is a tangential derivative. Note also that we have $n$ tangential derivatives $\de_1, \dots, \de_n$ and, thus, they are linearly dependent, since $\pa E$ is $(n-1)$-dimensional. However, it is easy to check (as an exercise) that
\begin{equation*}
|\de \xi|^2= \sum_{i=1}^n |\de_i \xi|^2 .
\end{equation*}

We next define {\it the Laplace-Beltrami operator} on $\pa E$ by
\begin{equation}\label{def:Laplace-Beltrami}
\DLB \xi := \sum_{i=1}^n \delta_i\delta_i \xi ,
\end{equation}
acting on functions $\xi: \pa E \to \RR$.
For the reader knowing Riemannian calculus, one can check that
$$
\DLB \xi = {\dv}_T (\na_T \xi) = {\dv}_T (\de \xi) , 
$$
where $\na_T=\de$ is the tangential gradient introduced in \eqref{def:tangentialgradient} and $\dv_T$ denotes the (tangential) divergence on the manifold $\pa E$.

According to \eqref{mean curvature formula}, we have that
$$
\cH = {\dv}_T \nu = \sum_{i=1}^n \de_i \nu^i.
$$ 
We will also use the following {\it formula of integration by parts}:
\begin{equation}\label{eq:Giustitypo}
\int_{\pa E} \de_i \phi \, dH_{n-1} = \int_{\pa E} \cH \phi \nu^i dH_{n-1}                                 
\end{equation}
for every (smooth) hypersurface $\pa E$ and $\phi \in C^1 (\pa E)$ with compact support.
Equation \eqref{eq:Giustitypo} is proved in Giusti's book \cite[Lemma 10.8]{G}. However, there are
%
%
two typos in \cite[Lemma 10.8]{G}: $\cH$ is missed in the identity above, and there is an error of a sign in the proof of \cite[Lemma 10.8]{G}.


Replacing $\phi$ by $\phi \fhi$ in \eqref{eq:Giustitypo}, we deduce that 
\begin{equation}\label{equa:intparts2}
\int_{\pa E} \phi \, \de_i \fhi \, dH_{n-1} = - \int_{\pa E} (\de_i \phi) \fhi \, dH_{n-1}  + \int_{\pa E} \cH \phi \fhi \nu^i dH_{n-1} .
\end{equation}
From this, replacing $\phi$ by $\de_i \phi$ in \eqref{equa:intparts2} and using that $\sum_{i=1}^n \nu^i \de_i \phi = \nu \cdot \de \phi=0$, we also have
\begin{equation}\label{equa:intLB}
\int_{\pa E} \de \phi \cdot \de \fhi \, dH_{n-1} = \sum_{i=1}^n \int_{\pa E} \de_i \phi \, \de_i \fhi \, dH_{n-1} = - \int_{\pa E} (\DLB \phi) \fhi \, dH_{n-1} .
\end{equation}
 
\begin{remark}\label{rem:Jacobi operator}
For a minimal surface $\partial E$, the second variation of perimeter given by \eqref{eq:minimal-1-2v2} can also be rewritten, after \eqref{equa:intLB}, as 
$$
\int_{\partial E} \left\{ -  \DLB \xi -c^2 \xi \right\} \xi \, dH_{n-1}.
$$
The operator $- \DLB -c^2$ appearing in this expression is called {\it the Jacobi operator}. It is the linearization at the minimal surface $\pa E$ of the minimal surface equation $\cH =0$. 
\end{remark}

Towards the proof of Simons theorem, let us now take $\xi=\tilde{c}\eta$ in 
\eqref{eq:1-2v2}, where $\tilde{c}$ and $\eta$ are still arbitrary ($\eta$ with 
compact support outside the origin) and will be chosen later.
We obtain
\begin{eqnarray*}
0 &\le& \int_{\partial E} \left\{ |\delta\xi|^2-c^2 \xi^2 \right\} dH_{n-1} \\
&=& \int_{\partial E} \left\{ \tilde{c}^2 |\delta \eta|^2 + \eta^2 | \delta \tilde{c} |^2
+ \tilde{c} \de \tilde{c} \cdot \de \eta^2 - c^2 \tilde{c}^2 \eta^2 \right\} dH_{n-1} \\
&=&  \int_{\partial E} \left\{  \tilde{c}^2|\delta\eta|^2 - ( \DLB \tilde{c} +c^2 \tilde{c} ) \tilde{c} \eta^2 \right\} dH_{n-1} ,
\end{eqnarray*}
where at the last step we used integration by parts \eqref{equa:intparts2}.
This leads to the inequality
\begin{equation*}
\int_{\partial E}  \left\{ \DLB \tilde{c} +c^2 \tilde{c} \right\} \tilde{c} \eta^2  dH_{n-1} \le  \int_{\partial E} \tilde{c}^2|\delta\eta|^2 dH_{n-1} ,
\end{equation*}
where the term $ \DLB \tilde{c} +c^2 \tilde{c}$ appearing in the first integral is the linearized or Jacobi operator at $\pa E$ acting on $\tilde{c}$.

Now we make the choice $\tilde{c} = c$ and we arrive, as a consequence of stability, to
\begin{equation}\label{eq:Disuguaglianza prima di Teo Simons}
\int_{\partial E} \left\{  \frac{1}{2} \DLB c^2 -  | \delta c|^2 +c^4 \right\} \eta^2  dH_{n-1} \le  \int_{\partial E} c^2|\delta\eta|^2 dH_{n-1} .
\end{equation}

At this point, Simons proof of Theorem \ref{thmcone} uses the following inequality for the Laplace-Beltrami operator $\De_{LB}$ of $c^2$
(recall that $c^2$ is the sum of the squares of the principal curvatures 
of $\partial E$), in the case
when $\partial E$ is a stationary cone.

\begin{lemma}[Simons lemma~\cite{S}]
\label{lmsimons}
Let $E\subset\RR^n$ be an open set such that $\partial E$
is a cone with zero mean curvature and $\partial E\setminus\{0\}$ is regular.
Then, $c^2$ is homogeneous of degree $-2$ and, in $\partial E\setminus\{0\}$, we have
\[
\frac{1}{2} \DLB c^2 - |\delta c|^2 + c^4  \geq \frac{2c^2}{|x|^2}.
\]
\end{lemma}

In Subsection \ref{subsec:SimLem} we will give an outline of the proof of this result.
We now use Lemma \ref{lmsimons} to complete the proof of Theorem \ref{thmcone}.

\begin{proof}[of Theorem \ref{thmcone}]
By using \eqref{eq:Disuguaglianza prima di Teo Simons} together with Lemma~\ref{lmsimons} we obtain
\begin{equation}
0\leq\int_{\partial
E}c^2\left\{|\delta\eta|^2-\frac{2\eta^2}{|x|^2}\right\}dH_{n-1}
\label{eq:1-mcstb}
\end{equation}
for every $\eta\in C^1(\partial E)$ with compact 
support outside the origin.
By approximation, the same holds for $\eta$ Lipschitz instead of $C^1$.


If $r=|x|$, we now choose $\eta$ to be the Lipschitz function
$$
\eta=
\bigg \{
\begin{array}{rl}
r^{- \al} \quad \text{if } r \le 1 \ \\
r^{- \be} \quad \text{if } r \ge 1 . \\
\end{array}
$$
By directly computing
\begin{equation}\label{equaz:provaetadelta}
|\de \eta|^2=
\bigg \{
\begin{array}{rl}
\al^2 r^{- 2 \al -2} \quad \text{if } r \le 1 \ \\
\be^2 r^{- 2 \be -2} \quad \text{if } r \ge 1 , \\
\end{array}
\end{equation}
we realize that if 
\begin{equation}\label{equaz:provaalfabeta}
\al^2<2 \, \mbox{ and } \, \be^2<2 ,
\end{equation}
then in \eqref{eq:1-mcstb} we have $|\de \eta|^2 - 2 \eta^2 / r^2 <0$.
If $\eta$ were an admissible function in \eqref{eq:1-mcstb}, we would then conclude that $c^2 \equiv 0$ on
$\pa E$. This is equivalent to $\pa E$ being an hyperplane.

Now, for $\eta$ to have compact support and hence be admissible, we need to cut-off $\eta$ near $0$ and infinity. As an exercise, one can check that the cut-offs work (i.e., the tails in the integrals tend to zero) if (and only if)
\begin{equation}\label{equazion:cutoffex}
\int_{\partial E}c^2 | \de \eta|^2  dH_{n-1}<\infty,
\end{equation}
or equivalently, since they have the same homogeneity,
$$
\int_{\partial E}c^2\frac{\eta^2}{|x|^2} \, dH_{n-1}<\infty.
$$
By recalling that the Jacobian on $\pa E$ (in spherical coordinates) is $r^{(n-1)-1}$, \eqref{equaz:provaetadelta}, and that, by Lemma~\ref{lmsimons}, $c^2$ is homogeneous of degree $-2$, we deduce that
%
%
\eqref{equazion:cutoffex} is satisfied if $n-6- 2\alpha  >-1$ and $n-6 - 2 \be <-1$. That is, if
\begin{equation}
\alpha < \frac{n -5 }{2} \quad\textrm{ and }\quad  \frac{n-5}{2} < \beta .
\label{eq:1-ineqs}
\end{equation}
%
%
If $3\le n\le 7$ then $(n -5 )^2/4<2$, i.e., $ - \sqrt{2} < (n -5) / 2 < \sqrt{2}$, and thus we can choose $\alpha$ and $\beta$ satisfying
\eqref{eq:1-ineqs} and \eqref{equaz:provaalfabeta}. It then follows that $c^2\equiv 0$, and
hence $\partial E$ is a hyperplane.
\qed
\end{proof}

The argument in the previous proof (leading to the dimension $n\le 7$) is very much related to a well known result: Hardy's inequality in $\RR^n$ -- which is presented next.

\subsection{Hardy's inequality}
As already noticed in Remark \ref{rem:Jacobi operator}, for a minimal surface $\partial E$ the second variation of perimeter \eqref{eq:minimal-1-2v2} can also be rewritten, by integrating by parts, as 
$$
\int_{\partial E} \left\{ -  \DLB \xi -c^2 \xi \right\} \xi dH_{n-1} .
$$
This involves the linearized or Jacobi operator $- \DLB -c^2$. If $x= |x| \si= r \si$, with
$\si \in S^{n-1}$, then $c^2= d(\si) / |x|^2$ (if $\pa E$ is a cone and thus $c^2$ is homogeneous of degree $-2$), where $d(\si)$ depends only on the angles $\si$. Thus, we are in the presence of the ``Hardy-type operator''
$$- \DLB - \frac{d(\si)}{|x|^2} ;$$
notice that $\DLB$ and $d(\si) / |x|^2 $ scale in the same way.
Thus, for all admissible functions $\xi$,
$$0 \le \int_{\partial E} \left\{ |\de \xi|^2 - \frac{d(\si)}{|x|^2} \, \xi^2 \right\} dH_{n-1} , \quad \mbox{if $\partial E$ is a stable minimal cone.}$$

Let us analyze the simplest case when $\pa E = \RR^n$ and $d \equiv constant$. Then, the validity or not of the previous inequality is given by Hardy's inequality, stated and proved next.

\begin{proposition}[Hardy's inequality]\label{Prop:Hardy inequality}
If $n \ge 3$ and $\xi \in C_{c}^{1}(\RR^n \setminus \lbrace 0 \rbrace)$, then
\begin{equation}\label{Hardyin}
\frac{(n-2)^{2}}{4}\int _{\RR^n}\frac{\xi ^{2}}{|x|^{2}} \,dx
\leq \int _{\RR^n}\left|\nabla \xi \right|^{2}dx .
\end{equation}
In addition, $(n-2)^{2} / 4$ is the best constant in this inequality and it is not achieved by any $0 \not\equiv \xi \in H^1 (\RR^n)$.

Moreover, if $a> (n-2)^{2} / 4$, then the Dirichlet spectrum of $- \DLB - a / |x|^2$  in the unit ball $B_1$ goes all the way to $- \infty$. That is,
\begin{equation}\label{equa:quotient}
\inf \frac{ \int_{B_1} \lbrace |\na \xi|^2 - a \, \frac{\xi^2}{|x|^2} \rbrace dx }{\int_{B_1} |\xi|^2 dx} = - \infty ,
\end{equation}
where the infimum is taken over $0 \not\equiv \xi \in H_0^1 (B_1)$.
\end{proposition}
\begin{proof}
Using spherical coordinates, for a given $\si \in S^{n-1}$ we can write
\begin{equation}\label{har1}
\int_0^{+\infty} r^{n-1} r^{-2} \xi^2(r \si) \, dr = - \frac{1}{n-2} \int_0^{+\infty} r^{n-2} 2 \xi (r \si) \xi_r (r \si) \, dr .
\end{equation}
Here we integrated by parts, using that $r^{n-3}= \left( r^{n-2} / (n-2) \right)'$.

Now we apply the Cauchy-Schwarz inequality in the right-hand side to obtain
\begin{multline}\label{har2}
- \int_0^{+\infty} r^{n-2} \xi \xi_r \, r^{\frac{n-3}{2} } \, r^{- \frac{n-3}{2} }  \, dr \le \left( \int_0^{+\infty} r^{n-3} \xi^2 \, dr\right)^{\frac{1}{2}} \, \left(\int_0^{+\infty} r^{n-1} \xi^2_r \, dr \right)^{\frac{1}{2}}.
\end{multline}
Putting together \eqref{har1} and \eqref{har2} we get
$$\int_0^{+\infty} r^{n-3} \xi^2 \, dr \le \frac{2}{n-2} \left( \int_0^{+\infty} r^{n-3} \xi^2 \, dr\right)^{\frac{1}{2}} \, \left(\int_0^{+\infty} r^{n-1} \xi^2_r \, dr \right)^{\frac{1}{2}},$$
that is,
$$\frac{(n-2)^2}{4} \int_0^{+\infty} r^{n-1} \, \frac{\xi^2 }{r^2} \, dr \le \int_0^{+\infty} r^{n-1} \xi^2_r \, dr . $$
By integrating in $\si$ we conclude \eqref{Hardyin}.
An inspection of the equality cases in the previous proof shows that the best constant is not achieved.

Let us now consider $(n-2)^2 / 4 < \al^2 < a$ with $\al \searrow (n-2) / 2$. Take
$$
\xi = r^{- \al} -1  
$$
and cut it off near the origin to be admissible. If we consider the main terms in the quotient \eqref{equa:quotient}, we get
$$
\frac{ \int (\al^2 - a) r^{-2 \al -2} dx }{ \int r^{-2 \al } dx } .
$$
Thus it is clear that, as $\al \searrow (n-2) / 2$, the denominator remains finite independently of the cut-off, while the numerator is as negative as we want after the cut-off. Hence, the quotient tends to $- \infty$.
\qed
\end{proof}

\begin{remark}\label{rem:dopohardy}
As we explained in Remark \ref{rem:Foliation B-DG-G}, in \cite{BdGG}, Bombieri, De Giorgi, and Giusti used a foliation made of exact solutions to the minimal surface equation $\cH=0$ when $2 m \ge 8$.
These are the solutions of the ODE \eqref{equa:odest} starting from points $\left( s(0),t(0) \right)= \left( s_0,0 \right) $ in the $s$-axis and with vertical derivative $\left( s'(0), t'(0) \right) = \left( 0,1 \right) $.
They showed that, for $2m \ge 8$, they do not intersect each other, neither intersect the Simons cone $\cSC$. Instead, in dimensions $4$ and $6$ they do not produce a foliation and, in fact, each of them crosses infinitely many times $\cSC$, as showed in Figure \ref{fig:add2}. This reflects the fact that the linearized operator  $- \DLB -c^2$ on $\cSC$ has infinitely many negative eigenvalues, as in the simpler situation of Hardy's inequality in the last statement of Proposition \ref{Prop:Hardy inequality}.
%
\end{remark}

\begin{figure}[htbp]
\centering
\includegraphics[scale=.25]{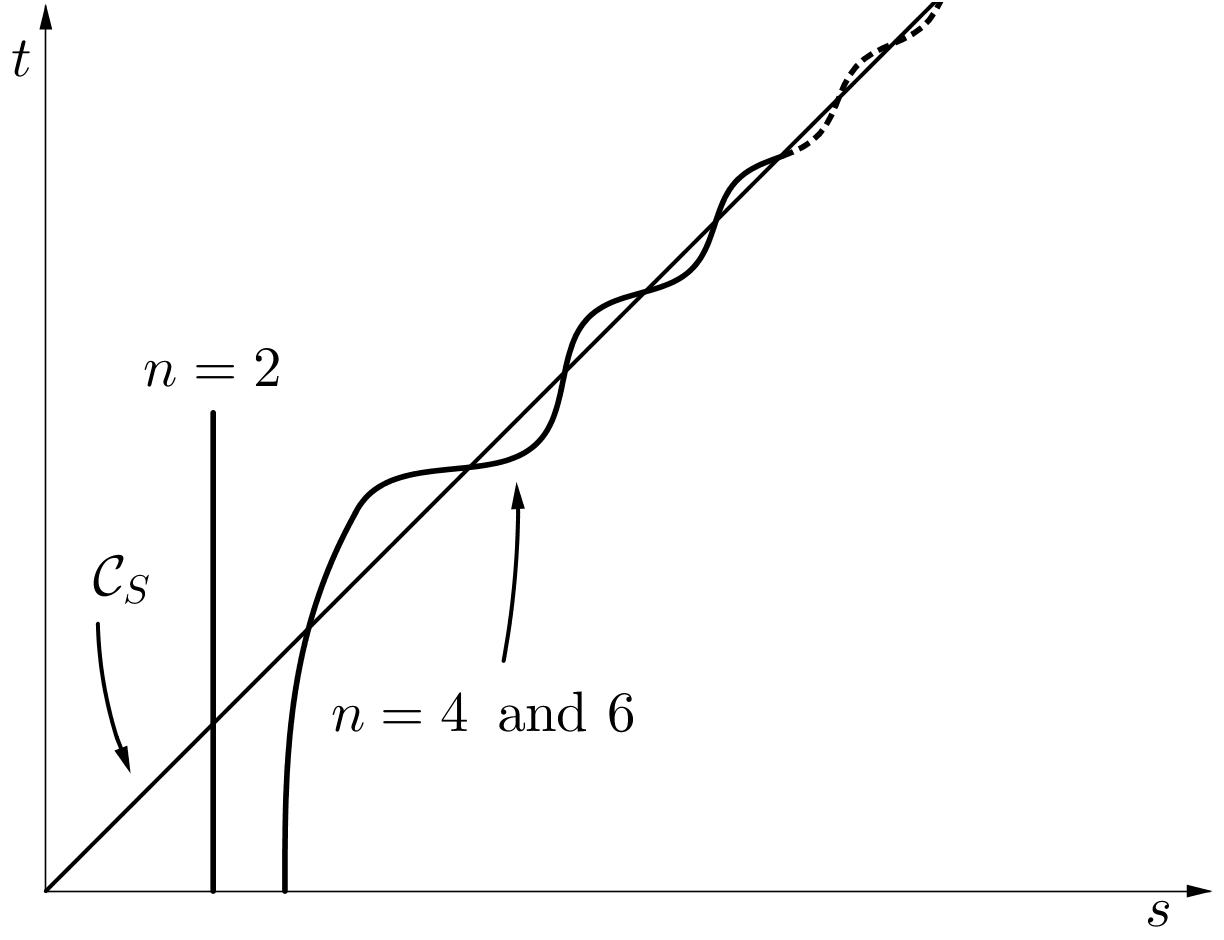}
\caption{Behaviour of the solutions to $\cH = 0$ in dimensions $2$, $4$, and $6$}
\label{fig:add2}       
\end{figure}

\subsection{Proof of the Simons lemma}\label{subsec:SimLem}
As promised, in this section we present the proof of Lemma \ref{lmsimons} with almost all details. 
We follow the proof contained in Giusti's book \cite{G}, where more details can be found (Simons lemma is Lemma 10.9 in \cite{G}). We point out that in the proof of \cite{G} there are the following two typos:
\begin{itemize}
\item as already noticed before, the identity in the statement of \cite[Lemma 10.8]{G} is missing $\cH$ in the second integrand. We wrote the corrected identity in equation \eqref{eq:Giustitypo} of these notes;
\item the label (10.18) is missing in line -8, page 122 of \cite{G}.
\end{itemize}

Alternative proofs of Lemma \ref{lmsimons} using intrinsic Riemaniann tensors can be found in the original paper of Simons \cite{S} from 1968 and also in the book of Colding and Minicozzi \cite{CM}.

\begin{notations}
We denote by $d(x)$ the signed distance function to $\pa E$, defined by
$$d(x) :=
\bigg \{
\begin{array}{rl}
\operatorname{dist} (x, \pa E) , & \ \ x \in \RR^n \setminus E , \\
- \operatorname{dist} (x, \pa E) , & \ \ x \in E . \\
\end{array}
$$
As we are assuming $E \setminus \{ 0 \}$ to be regular, we have that $d(x)$ is $C^2$ in a neighborhood of $\pa E \setminus \{ 0 \} $.

The normal vector to $\pa E$ is given by 
$$\nu= \na d= \frac{ \na d}{| \na d |};$$
we write
$$\nu= (\nu^1, \dots, \nu^n) = (d_1, \dots, d_n) ,$$
where we adopt the abbreviated notation
$$w_i=w_{x_i} = \pa_i w \, \mbox{ and } \, w_{ij}=w_{x_i x_j} = \pa_{ij} w $$
for partial derivatives in $\RR^n$.
As introduced after Theorem \ref{thmcone}, we will use the tangential derivatives
\begin{equation*}
\de_i := \pa_i - \nu^i \nu^k \pa_k
\end{equation*}
for $i=1, \dots, n$, and thus
$$
\de_i w = w_i - \nu^i \nu^k w_k,
$$
where we adopted the summation convention over repeated indices.
Finally, recall the Laplace-Beltrami operator defined in \eqref{def:Laplace-Beltrami}:
\begin{equation*}
\DLB := \de_i \de_i. 
\end{equation*}

\end{notations}


\begin{remark}
Since
\begin{equation}\label{equazione:nu2}
1= | \nu |^2 = \sum_{k=1}^n d_k^2 ,
\end{equation}
it holds that
$$
d_{jk} d_k =0 \, \mbox{ for } j=1, \dots,n.
$$
Thus, we have
$$
\de_i \nu^j= \de_i d_j = d_{ij} - d_i d_k d_{kj}= d_{ij}=d_{ji},
$$
which leads to
\begin{equation*}
\de_i \nu^j= \de_j \nu^i .
\end{equation*}
\end{remark}

\begin{exercise}
Using $\de_i \nu^j = d_{ij}$, verify that
\begin{equation*}
\cH= \de_i \nu^i ,
\end{equation*}
\begin{equation}
\label{eq:def c indici}
c^2= \de_i \nu^j \de_j \nu^i = \sum_{i,j=1}^n (\de_i \nu^j)^2 .
\end{equation}
\end{exercise}

The identities
\begin{equation*}
\nu^i \de_i = 0 ,
\end{equation*}
\begin{equation}
\label{eq:utileperSimonslemmaproof}
\nu^i \de_j \nu^i = 0 , \, \mbox{ for } j=1,\dots,n ,
\end{equation}
will be used often in the following computations. The first one follows from the definition of $\de_i$, while the second is immediate from \eqref{equazione:nu2}.

The next lemma will be useful in what follows.
\begin{lemma}
The following equations hold for every smooth hypersurface $\pa E$:
\begin{equation}
\label{eq:Giusti 1}
\de_i \de_j = \de_j \de_i + (\nu^i \de_j \nu^k - \nu^j \de_i \nu^k)\de_k ,
\end{equation}
\begin{equation}
\label{eq:Giusti 2}
\DLB \nu^j + c^2 \nu^j = \de_j \cH  \, (=0 \, \mbox{ if $\pa E$ is stationary}) ,
\end{equation}
for all indices $i$ and $j$.
\end{lemma} 
For a proof of this lemma, see \cite[Lemma 10.7]{G}.

Equation \eqref{eq:Giusti 2} is an important one. It says that the normal vector $\nu$ to a minimal surface solves the Jacobi equation $\left( \DLB + c^2 \right) \nu \equiv 0$ on $\pa E$. This reflects the invariance of the equation $\cH = 0$ by translations (to see this, write a perturbation made by a small translation as a normal deformation, as in Figure \ref{fig:2}).

If $\pa E$ is stationary, from \eqref{eq:Giusti 1} and by means of simple calculations, one obtains that
\begin{equation}
\label{eq:Giusti 3}
\DLB \de_k = \de_k \DLB - 2 \nu^k (\de_i \nu^j) \de_i \de_j - 2 (\de_k \nu^j)(\de_j \nu^i) \de_i .
\end{equation}
Equation \eqref{eq:Giusti 3} is the formula with the missed label (10.18) in \cite{G}.


We are ready now to give the

\begin{sketch}[of Lemma \ref{lmsimons}]
By \eqref{eq:def c indici} we can write that
$$
\frac{1}{2} \DLB c^2 = (\de_i \nu^j) \DLB \de_i \nu^j + \sum_{i,j,k} (\de_k \de_i \nu^j)^2.
$$
Then, using \eqref{eq:Giusti 2}, \eqref{eq:Giusti 3}, and the fact $\cH=0$, we have
$$
\frac{1}{2} \DLB c^2 = - (\de_i \nu^j) \de_i (c^2 \nu^j) - 2 (\de_i \nu^j) (\de_k \nu^l) (\de_l \nu^j) (\de_i \nu^k) + \sum_{i,j,k} (\de_k \de_i \nu^j)^2 ,
$$
and by \eqref{eq:Giusti 1}
$$
\frac{1}{2} \DLB c^2  = -c^4 - 2 \nu^i \nu^l (\de_j \de_l \nu^k) (\de_k \de_i \nu^j) + \sum_{i,j,k} (\de_k \de_i \nu^j)^2.
$$
Now, if $x_0 \in \pa E \setminus \{ 0 \}$, we can choose the $x_n$-axis to be the same direction as $\nu(x_0)$. Thus, $\nu(x_0)=(0,\dots,0,1)$ and at $x_0$ we have
$$
\nu^n=1 , \, \de_n=0 ,
$$
$$
\nu^{\al}=0 , \, \de_\al= \pa_\al \, \mbox{ for } \al=1,\dots,n-1 .
$$
Hence, computing from now on always at the point $x_0$, and by using \eqref{eq:Giusti 1} and \eqref{eq:utileperSimonslemmaproof}, we get
\begin{eqnarray*}
\frac{1}{2} \DLB c^2 &=& -c^4 + \sum_{\al,\be,\ga} (\de_\ga \de_\al \nu^\be)^2 + 2 \sum_{\al, \ga} (\de_\ga \de_\al \nu^n)^2 - 2 \sum_{\al, \be} (\de_\al \de_\be \nu^n)^2
\\
&=& -c^4 + \sum_{\al,\be,\ga} (\de_\ga \de_\al \nu^\be)^2 ,
\end{eqnarray*}
where all the greek indices indicate summation from $1$ to $n-1$.

On the other hand, we have 
$$
| \de c |^2 = \frac{1}{c^2} (\de_\al \nu^\be) (\de_\ga \de_\al \nu^\be) (\de_\si \nu^\tau) (\de_\ga \de_\si \nu^\tau) ,
$$
and hence
\begin{equation*}
\frac{1}{2} \DLB c^2 + c^4 - | \de c |^2 = \frac{1}{2 c^2} \sum_{\al, \be, \ga, \si, \tau} \left[ (\de_\si \nu^{\tau}) (\de_\ga \de_\al \nu^\be) - (\de_\al \nu^\be) (\de_\ga \de_\si \nu^\tau) \right]^2 .
\end{equation*}

Now remember that $\pa E$ is a cone with vertex at the origin, and thus
$<x, \nu> = 0$ on $\pa E$.
Since we took $\nu= (0, \dots , 0 , 1)$ at $x_0$, we may choose coordinates in such a way that $x_0$ lies on the $(n-1)$-axis. In particular, $\nu^{n-1} = 0$ at $x_0$ and
$$
0 = \de_i <x, \nu> = <\de_i x, \nu>+<x, \de_i \nu> = <x, \de_i \nu>,
$$
which leads to
$$
\de_i \nu^{n-1} = 0 \, \mbox { at } \, x_0 .
$$

If now the letters $A,B,S,T$ run from $1$ to $n-2$, he have
\begin{equation*}
\begin{split}
\frac{1}{2} \DLB c^2 + c^4 - | \de c |^2 = & \frac{1}{2 c^2} \sum_{A,B,S,T, \ga} \left[ (\de_S \nu^T) (\de_\ga \de_A \nu^B) - (\de_A \nu^B) (\de_\ga \de_S \nu^T) \right]^2 
\\
& + \frac{2}{c^2} \sum_{S,T, \ga, \al} (\de_S \nu^T)^2 (\de_\ga \de_{n-1} \nu^{\al})^2  \ge 2 \sum_{\al, \ga} ( \de_\ga \de_{n-1} \nu^{\al} )^2.
\end{split}
\end{equation*}
From \eqref{eq:Giusti 1}, $\de_i \de_{n-1} = \de_{n-1} \de_i $ and $\de_{n-1} = \pa_{n-1} = \pm \left( x^j /|x| \right) \pa_j$ at $x_0$.
Since $\pa E$ is a cone, $\nu$ is homogeneous of degree 0 and hence $\de_i \nu^\al$ is homogeneous of degree $-1$. Thus, by Euler's theorem on homogeneous functions, we have
$$
\de_{n-1} \de_i \nu^{\al} = \pm  \frac{x^j \pa_j }{|x|}  \de_i \nu^{\al} = \mp \frac{1}{|x|} \de_i \nu^{\al} ,
$$ 
and hence
$$
2 \sum_{i, \al } ( \de_i \de_{n-1} \nu^{\al} )^2 = \frac{2}{ | x |^2} \sum _{i, \al} (\de_i \nu^\al)^2 = \frac{2 c^2}{| x |^2} .
$$
The proof is now completed.
\qed
\end{sketch}

\subsection{Comments on: harmonic maps, free boundary problems, and nonlocal minimal surfaces}
Here we briefly sketch arguments and results similar to the previous ones on minimal surfaces, now for three other elliptic problems.

\subsubsection{Harmonic maps}
Consider the energy
\begin{equation}
\label{eq:harm_enrg}
E(u)=\frac{1}{2}\int_{\Omega}|Du|^2dx
\end{equation}
for $H^1$ maps $u:\Omega\subset\RR^n \to \overline{S^N_+} $ from a domain
$\Omega$ of $\RR^n$ into the closed upper hemisphere
$$\overline{S^N_+}=\{y\in \mathbb{R}^{N+1}: |y|=1, y_{N+1}\ge 0\}.$$

A critical point of $E$ is called a (weakly) {\it harmonic map}.
When a map minimizes $E$ among all maps with values into $\overline{S^N_+}$ and with same boundary
values on $\partial \Omega$, then it is called a 
minimizing harmonic map.

From the energy \eqref{eq:harm_enrg} and the restriction $|u| \equiv 1$, one finds that the equation for harmonic maps is given by
\begin{equation*}
-\Delta u=|Du|^2u \qquad\text{in }\Omega.
\end{equation*}

In 1983, J\"ager and Kaul proved the following theorem, that we state here without proving it (see the original paper \cite{JK} for the proof).

\begin{theorem}[J\"ager-Kaul~\cite{JK}]\label{thm:equator}
The equator map 
\begin{equation*}
u_*:B_1\subset\RR^n \rightarrow \overline{S^n_+},
\quad x\mapsto \left(x/|x|,0\right)
\end{equation*}
is a minimizing harmonic map on the class
$$
{\mathcal C}=\{u\in H^1(B_1\subset\RR^n,S^n) : u=u_* \text{ on }\partial B_1\}
$$
if and only if $n\ge 7$.
\end{theorem}
We just mention that the proof of the ``if" in Theorem \ref{thm:equator} uses a calibration argument.

Later, Giaquinta and Sou\v cek \cite{GS}, and independently Schoen and Uhlenbeck \cite{SU2}, proved the following result.

\begin{theorem}[Giaquinta-Sou\v cek \cite{GS}; Schoen-Uhlenbeck \cite{SU2}]\label{thm:reg_harm}
\hfill\\
Let $u:B_1\subset\RR^n\to \overline{S^N_+}$ be a minimizing harmonic map,
homogeneous of degree zero, 
into the closed upper hemisphere $\overline{S^N_+}$. If $3\le n \le 6$,
then $u$ is constant.
\end{theorem}

Now we will show an outline of the proof of Theorem~\ref{thm:reg_harm} following~\cite{GS}. More details can also be found in Section 3 of \cite{CabCapThree}. This theorem gives an alternative proof of one part of the statement of Theorem~\ref{thm:equator}. Namely, that the equator map $u_*$ 
is not minimizing for $3\le n\le 6$.

\begin{sketch}[of Theorem~\ref{thm:reg_harm}]
After stereographic projection (with respect to the south pole)
$P$ from $S^{N}\subset\RR^{N+1}$ to $\RR^N$, for the new
function $v=P\circ u: B_1\subset\RR^n\to\RR^N$,
the energy \eqref{eq:harm_enrg} (up to a constant factor) 
is given by 
\begin{equation*}
E(v):=\int_{B_1}\frac{|Dv|^2}{(1+|v|^2)^2}dx.
\end{equation*}
In addition, we have $|v|\le 1$ since the image of $u$ is contained
in the closed upper hemisphere.

By testing the function
$$
\xi(x)=v(x)\eta(|x|),
$$ 
where 
$\eta$ is a smooth radial function with compact support in $B_1$, in the equation of the first variation of the energy, that is
\begin{equation*}
\delta E(v) \xi=0 ,
\end{equation*}
one can deduce that either $v$ is constant (and 
then the proof is finished) or
$$
|v|\equiv 1 ,
$$
that we assume from now on.

Since $v$ is a minimizer, we have that the second variation of the energy satisfies
$$\delta^2 E(v) (\xi,\xi) \ge 0 .$$
By choosing here the function 
$$
\xi(x)=v(x)|Dv(x)|\eta(|x|),
$$ 
where 
$\eta$ is a smooth radial function with compact support in $B_1$
(to be chosen later), and setting
$$
c(x):=|Dv(x)|,
$$
one can conclude the proof by similar arguments as in the previous section and by using Lemma \ref{lem:harm}, stated next.
\qed
\end{sketch}

\begin{lemma}\label{lem:harm}
If $v$ is a harmonic map, homogeneous of degree
zero, and with $|v|\equiv 1$, we have
\begin{equation*}
\frac{1}{2}\Delta c^2-|Dc|^2 + c^4 \geq \frac{c^2}{|x|^2}+\frac{c^4}{n-1},
\end{equation*}
where $c:=|Dv|$.
\end{lemma}

This lemma is the analogue result of Lemma~\ref{lmsimons} for
minimal cones. See \cite{GS} for a proof of the lemma,
which also follows from Bochner identity (see \cite{SU2}).

\subsubsection{Free boundary problems}
Consider the one-phase free boundary problem:
%
\begin{equation}\label{000:eq:freboundpb}
\left\{
\begin{array}{rcll}
\De u &=& 0 & \quad\mbox{in } E \\
u &=& 0  & \quad\mbox{on } \pa E \\
| \na u | &=& 1  & \quad\mbox{on } \pa E \setminus \lbrace 0 \rbrace ,\\
\end{array}\right.
\end{equation}
where $u$ is homogeneous of degree one and positive in the domain $E \subset \RR^n$ and $\pa E$ is a cone.
We are interested in solutions $u$ that are stable for the Alt-Caffarelli energy functional
\begin{equation*}
E_{B_1}(u) = \int_{B_1} \left\lbrace | \na u |^2 + \mathbbm{1}_{ \left\lbrace u> 0 \right\rbrace  } \right\rbrace  dx
\end{equation*}
with respect to compact domain deformations that do not contain the origin. More precisely, we say that $u$ is {\it stable} if for any smooth vector field $\Psi: \RR^n \rightarrow \RR^n$ with $0 \notin {\rm supp} \Psi \subset B_1$ we have
\begin{equation*}
\left.\frac{d^2}{dt^2} E_{B_1} \left(u \left(x + t \Psi(x) \right) \right) \right|_{t=0} \ge 0.
\end{equation*}
The following result due to Jerison and Savin is contained in \cite{JS}, where a detailed proof can be found.

\begin{theorem}[Jerison-Savin \cite{JS}]\label{JSfreebp}
The only stable, homogeneous of degree one, solutions of \eqref{000:eq:freboundpb} in dimension $n\le 4$ are the one-dimensional solutions $u=(x \cdot \nu )^+$, $\nu \in S^{n-1}$.
\end{theorem}

In dimension $n=3$ this result had been established by Caffarelli, Jerison, and Kenig \cite{CJK}, where they conjectured that it remains true up to dimension $n \le 6$.
On the other hand, in dimension $n=7$, De Silva and Jerison \cite{DJ} provided an example of a nontrivial minimizer.

The proof of Jerison and Savin of Theorem \ref{JSfreebp} is similar to Simons proof of the rigidity of stable minimal cones in low dimensions: they find functions $c$ (now involving the second derivatives of $u$) which satisfy appropriate differential inequalities for the linearized equation. 

Here, the linearized problem is the following:
%
$$
\left\{
\begin{array}{rcll}
\De v &=& 0 & \quad\mbox{in } E \\
v_{\nu} + \cH v &=& 0  & \quad\mbox{on } \pa E \setminus \left\lbrace 0 \right\rbrace . \\
\end{array}\right.
$$
For the function
$$
c^2= \nr D^2 u \nr^2 = \sum_{i,j =1}^n u_{ij}^2 ,
$$
they found the following interior inequality which is similar to the one of the Simons lemma:
$$\frac{1}{2} \De c^2 - |\na c|^2 \geq 2 \, \frac{n-2}{n-1} \, \frac{c^2}{|x|^2}+ \frac{2}{n-1} \, |\na c|^2.$$
In addition, they also need to prove a boundary inequality involving $c_\nu$.
Furthermore, to establish Theorem \ref{JSfreebp} in dimension $n=4$, a more involved function $c$ of the second derivatives of $u$ is needed.

\subsubsection{Nonlocal minimal surfaces}
Nonlocal minimal surfaces, or $\al$-minimal surfaces (where $\al \in (0,1)$), have been introduced in 2010 in the seminal paper of Caffarelli, Roquejoffre, and Savin \cite{CRS}. These surfaces are connected to fractional perimeters and diffusions and, as $\al \nearrow 1$, they converge to classical minimal surfaces.
We refer to the lecture notes \cite{CF} and the survey \cite{DipVal}, where more references can be found.

For $\al$-minimal surfaces and all $\al \in (0,1)$, the analogue of Simons flatness result is only known in dimension $2$ by a result for minimizers of Savin and Valdinoci \cite{SV}.

\section{The Allen-Cahn equation}\label{ALLENCA}
This section concerns the {\it Allen-Cahn equation}
\begin{equation}\label{AlCa}
- \Delta u= u - u^3 \quad \mbox{in $\RR^{n}$}.
\end{equation}
By using equation \eqref{AlCa} and the maximum principle it can be proved that any solution satisfies
$|u| \le 1$. Then, by the strong maximum principle we have that either $|u|<1$  or $u \equiv \pm 1$.
Since $u \equiv \pm 1$ are trivial solutions, from now on we consider $u:\RR^{n}\to (-1,1)$.

We introduce the class of {\it 1-d solutions}: 
$$
u(x)= u_* (x \cdot e) \quad \mbox{ for a vector } \, e \in \RR^n, |e|=1,
$$ 
where
$$
u_* (y)= \tanh\left(\frac{y}{\sqrt{2}} \right).
$$
%
%
%
The solution $u_*$ is sometimes referred to as the {\it layer solution} to \eqref{AlCa};
see Figure \ref{fig:8}.
The fact that $u$ depends only on one variable can be rephrased also by saying that all the level sets
$\{u=s\}$ of $u$ are
hyperplanes.

\begin{figure}[htbp]
\centering
\includegraphics[scale=.25]{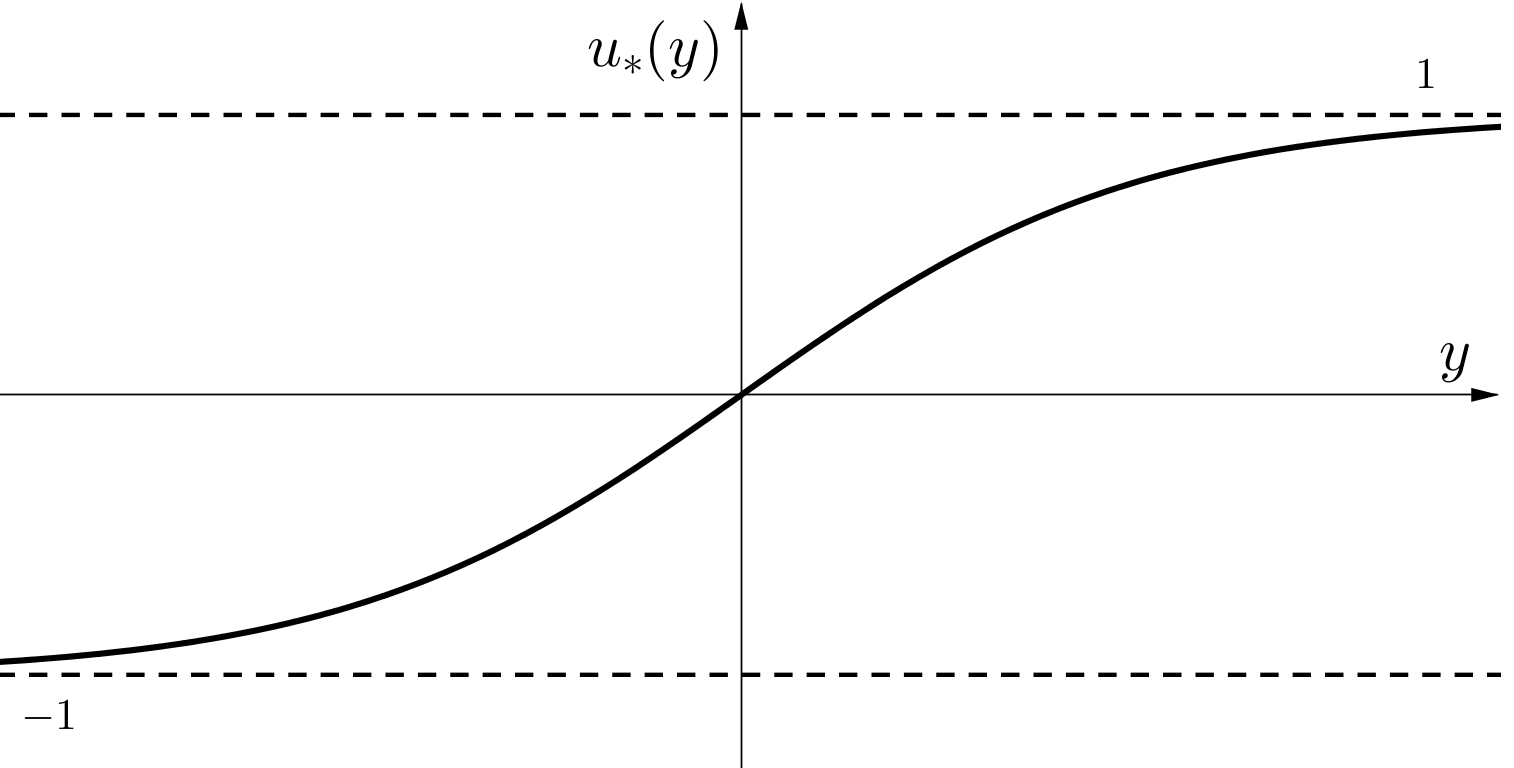}
\caption{The increasing, or layer, solution to the Allen-Cahn equation}
\label{fig:8}       
\end{figure}

\begin{exercise}
Check that the $1$-d functions introduced above are solutions of the Allen-Cahn equation.
\end{exercise}

\begin{remark}\label{rem:1dthenconditions}
Let us take $e= e_n = (0, \dots, 0,1)$ and consider the 1-d solution
$u(x)=u_* (x_n) = \tanh ( x_n / \sqrt{2} )$. It is clear that the following two relations hold:
\begin{equation}\label{monoton}
u_{x_n} >0 \quad \mbox{in }\RR^{n} ,
\end{equation}

\begin{equation}\label{limiting}
\lim_{x_n\rightarrow\pm\infty}u(x', x_n)=\pm 1
\quad \text{ for all }x'\in \RR^{n-1}.
 \end{equation}

\end{remark}

The energy functional associated to equation \eqref{AlCa} is
\begin{equation*}
E_\Om (u):=\int_\Omega \left\{ \frac{1}{2} |\nabla u|^2
+ G(u)\right\} dx ,
\end{equation*}
where $G$ is the double-well potential in Figure \ref{fig:7}:
\begin{equation*}
G(u)= \frac{1}{4} \left( 1- u^2 \right)^2. 
\end{equation*}

\begin{definition}[Minimizer]\label{definit:MinimizAC}
A function  $u:\RR^{n}\to (-1,1)$ is said to be a {\it minimizer} of \eqref{AlCa}
when
$$E_{B_R} (u) \leq E_{B_R} (v)$$
for every open ball $B_R$ and functions $v: \ol{B_R}\to \RR$ such that $v\equiv u$ on 
$\partial B_R$.
\end{definition}

\begin{figure}[htbp]
\centering
\includegraphics[scale=.25]{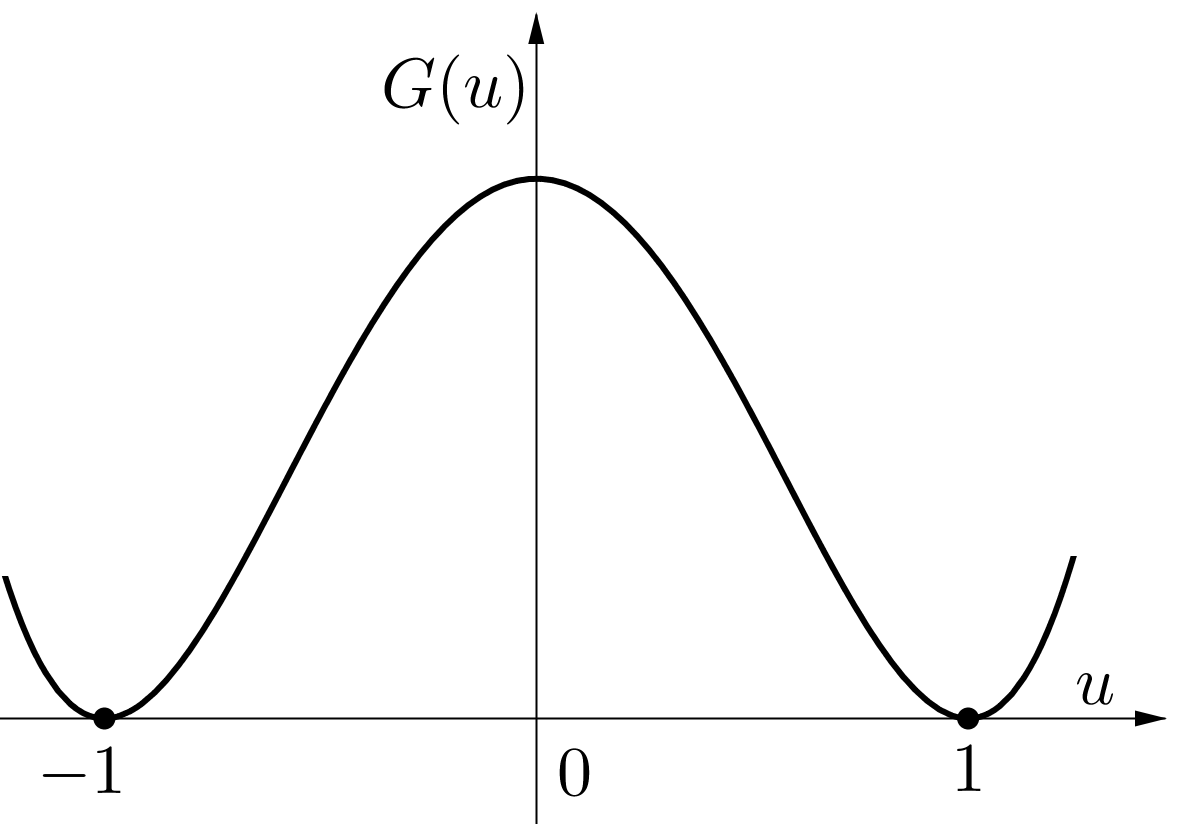}
\caption{The double-well potential in the Allen-Cahn energy}
\label{fig:7}       
\end{figure}

\smallskip
\noindent
{\bf Connection with the theory of minimal surfaces.}
The Allen-Cahn equation has its origin in the theory of phase transitions and it is used as a model for some nonlinear reaction-diffusion processes.
To better understand this, let $\Om\subset \RR^n$ be a bounded domain, and consider the Allen-Cahn equation with parameter $\ve > 0$,
\begin{equation}\label{eq:Allen-Cahnwithparameter}
- \ve^2 \De u= u - u^3 \, \mbox{ in } \, \Om ,
\end{equation}
with associated energy functional given by
\begin{equation}\label{eq:energyACwithparameter}
E_{\ve}(u)= \int_{\Om} \left\lbrace \frac{\ve}{2} \, | \na u|^2 + \frac{1}{\ve} \, G(u) \right\rbrace dx .
\end{equation}
Assume now that there are two populations (or chemical states) A and B and that $u$ is a density measuring the percentage of the two populations at every point: if $u(x) = 1$ (respectively, $u(x)= -1$) at a point $x$, we have only population A at $x$ (respectively, population B); $u(x)=0$ means that at $x$ we have $50\%$ of population A and $50\%$ of population $B$.

By \eqref{eq:energyACwithparameter}, it is clear that in order to minimize $E_{\ve}$ as $\ve$ tends to $0$, $G(u)$ must be very small. From Figure \ref{fig:7} we see that this happens when $u$ is close to $\pm 1$. These heuristics are indeed formally confirmed by a celebrated theorem of Modica and Mortola. It states that, if $u_\ve$ is a family of minimizers of $E_\ve$,
then, up to a subsequence, $u_\ve$ converges in $L^1_{\mathrm{loc}} (\Om)$, as $\ve$ tends to $0$,
to
$$
u_0 = \mathbbm{1}_{ \Om_+  } - \mathbbm{1}_{ \Om_-  } 
$$
for some disjoint sets $\Om_{\pm}$ having as common boundary a surface $\Ga$. In addition, $\Ga$ is a minimizing minimal surface.
Therefore, the result of Modica-Mortola establishes that the two populations tend to their total separation, and in such a (clever) way that the interface surface $\Ga$ of separation has least possible area.

Finally, notice that the $1$-d solution of \eqref{eq:Allen-Cahnwithparameter},
$$
x \mapsto u_* (\frac{x \cdot e}{\ve}) ,
$$
makes a very fast transition from $-1$ to $1$ in a scale of order $\ve$. Accordingly, in Figure \ref{fig:add1}, $u_\ve$ will make this type of fast transition across the limiting minimizing minimal surface $\Ga$.   
The interested reader can see \cite{AAC} for more details.

\smallskip

\begin{figure}[htbp]
\centering
\includegraphics[scale=.25]{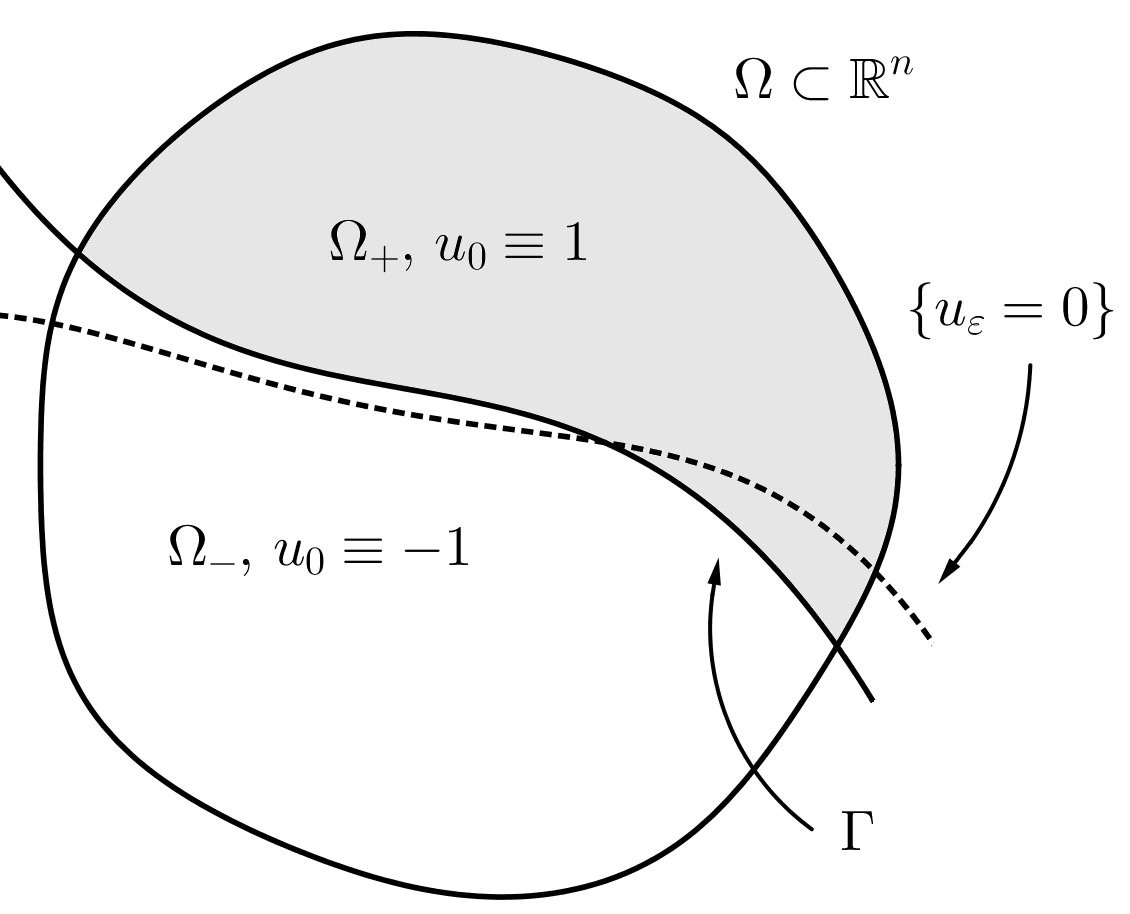}
\caption{The zero level set of $u_\ve$, the limiting function $u_0$, and the minimal surface $\Ga$}
\label{fig:add1}       
\end{figure}

\subsection{Minimality of monotone solutions with limits $\pm 1$}

The following fundamental result shows that monotone solutions with limits $\pm 1$ are minimizers (as in Definition \ref{definit:MinimizAC}).

\begin{theorem}[\textbf{Alberti-Ambrosio-Cabr\'e \cite{AAC}}]\label{alba}
Suppose that $u$ is a solution of \eqref{AlCa} satisfying the monotonicity hypothesis \eqref{monoton} and the condition \eqref{limiting} on limits.
Then, $u$ is a minimizer of \eqref{AlCa} in $\RR^n$.
\end{theorem}

See \cite{AAC} for the original proof of the Theorem~\ref{alba}. It uses a calibration built from a foliation and avoids the use of the strong maximum principle, but it is slightly involved. Instead, the simple proof that we give here was suggested to the first author (after one of his lectures on \cite{AAC}) by  L.~Caffarelli.
It uses a simple foliation argument together with the strong maximum principle, as in the alternative proof of Theorem \ref{thm:SimCone} given in Subsection \ref{subsec 1.1:MinimalitySimcones}.

\begin{proof}[of Theorem \ref{alba}]
Denoting $x= (x' ,x_n) \in \RR^{n-1} \times \RR$, let us consider the functions
$$
u^t (x):= u(x', x_n +t), \, \mbox{ for } \, t\in\RR.
$$
By the monotonicity assumption \eqref{monoton} we have that
\begin{equation}\label{eq:monfoliation}
u^t < u^{t'} \, \mbox{ in } \RR^n , \, \mbox{ if } t<t'.
\end{equation}
Thus, by \eqref{limiting} we have that the graphs of $u^t= u^t (x)$, $t \in \RR$,
form a foliation filling all of $\RR^n \times (-1,1)$.
Moreover, we have that for every $t \in \RR$, $u^t$ are solutions of $-\De u^t =u^t - (u^t)^3$ in $\RR^n$.

By simple arguments of the Calculus of Variations, given a ball $B_R$ it can be proved that there exists
a minimizer $v: \ol{B}_R \to \RR$ of $E_{B_R}$ such that $v=u$ on $\pa B_R$. In particular, $v$ satisfies
%
%
%
\begin{equation*}
\left\{
\begin{array}{rcll}
- \De v &=& v - v^3 & \quad\mbox{in } B_R \\
|v| &<& 1  & \quad\mbox{in } \ol{B}_R \\
v &=& u  & \quad\mbox{on } \pa B_R .\\
\end{array}\right.
\end{equation*}

By \eqref{limiting}, we have that the graph of $u^t$ in the compact set $\ol{B}_R$ is above the graph of $v$ for $t$ large enough, and it is below the graph of $v$ for $t$ negative enough
(see Figure~\ref{fig:9}).
If $v \not\equiv u$, assume that $v<u$ at some point in $B_R$ (the situation $v>u$ somewhere in $B_R$ is done similarly). It follows that, starting from $t= - \infty$, there will exist a first $t_* <0$
such that $u^{t_*}$ touches $v$ at a point $P \in \ol{B}_R$. This means that
$u^{t_*} \le v$ in $\ol{B}_R$ and $u^{t_*} (P) = v(P)$.

\begin{figure}[htbp]
\centering
\includegraphics[scale=.25]{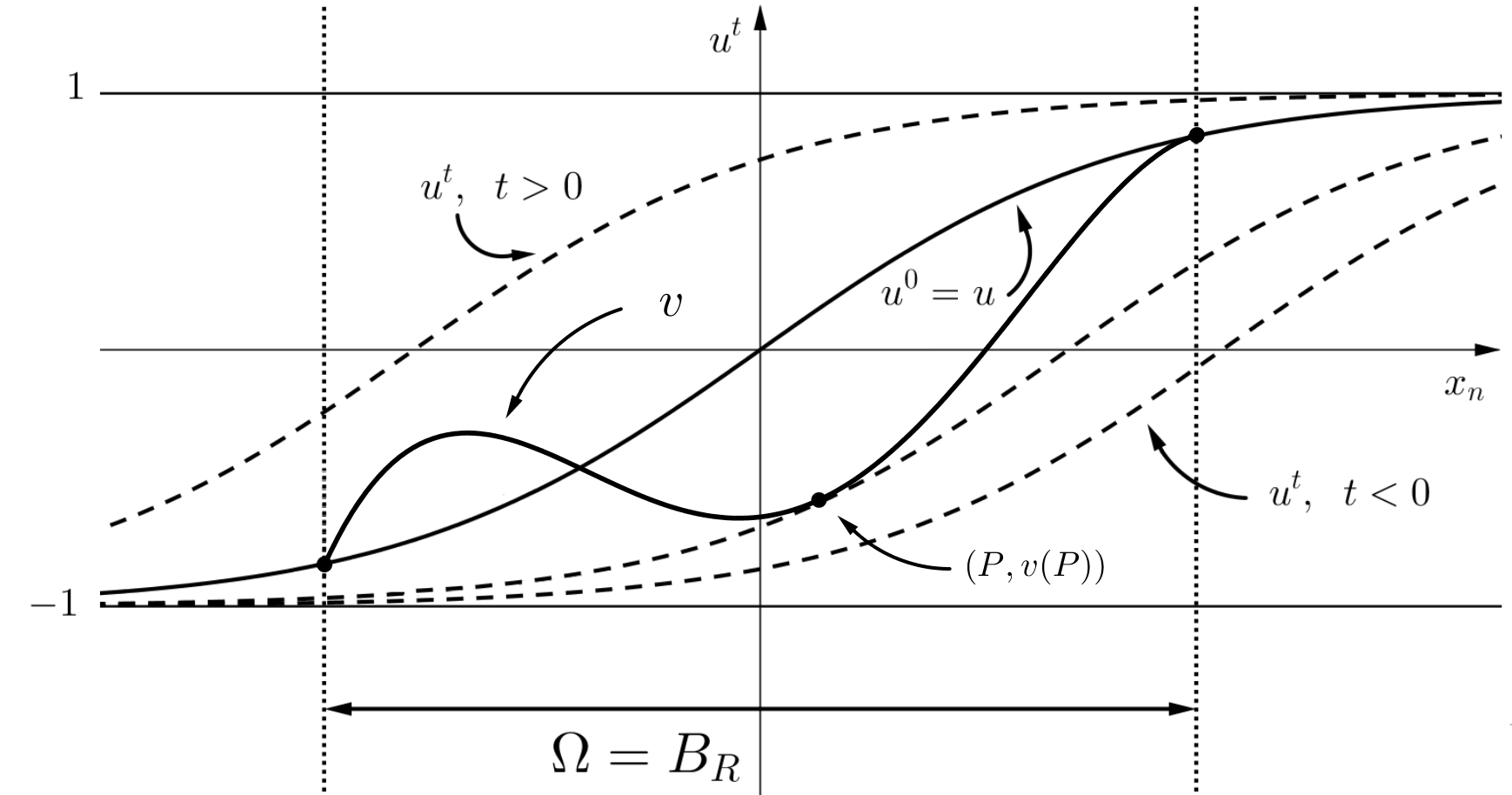}
\caption{The foliation $\lbrace u^t \rbrace$ and the minimizer $v$}
\label{fig:9}       
\end{figure}

By \eqref{eq:monfoliation}, $t_* <0$, and the fact that $v=u = u^0$ on $\pa B_R$, the point $P$ cannot belong to
$\pa B_R$. Thus, $P$ will be an interior point of $B_R$.

But then we have that $u^{t_*}$ and $v$ are two solutions of the same semilinear equation (the Allen-Cahn equation), the graph of $u^{t_*}$ stays below that of $v$, and they touch each other at the interior point $\left( P, v(P) \right) $. This is a contradiction with the strong maximum principle.

Here we leave as an exercise (stated next) to verify that the difference of two solutions of $- \De u = f(u)$ satisfies a linear elliptic equation to which we can apply the strong maximum principle.
This leads to $u^{t_*} \equiv v$, which contradicts $u^{t_*} < v = u^0$ on $\pa B_R$.
\qed
\end{proof}

\begin{exercise}
Prove that the difference $w:= v_1 - v_2$ of two solutions of a semilinear equation $- \De v= f(v)$, where $f$ is a Lipschitz function, satisfies a linear equation of the form $\De w + c(x) \, w = 0$, for some function $c \in L^{\infty}$. Verify that, as a consequence, this leads to $u^{t_*} \equiv v$ in the previous proof.
\end{exercise}

By recalling Remark \ref{rem:1dthenconditions}, we immediately get the following corollary.
\begin{corollary}
The 1-d solution $u(x)=u_* (x \cdot e)$ is a minimizer of \eqref{AlCa} in $\RR^n$, for every unit vector $e \in \RR^n$. 
\end{corollary}


As a corollary of Theorem \ref{alba}, we easily deduce the following important energy estimates.

\begin{corollary}[Energy upper bounds; \textbf{Ambrosio-Cabr\'e \cite{ambcab}} ]\label{tupper}
Let $u$ be a solution of \eqref{AlCa} satisfying
\eqref{monoton} and \eqref{limiting} (or more generally, let $u$ be a minimizer in $\RR^n$).

Then, for all $R \ge 1$ we have
\begin{equation}\label{upperbou}
E_{B_R} (u) \leq C R^{n-1}
\end{equation}
for some constant $C$ independent of $R$.
In particular, since $G \ge 0$, we have that
$$
\int_{B_R} |\nabla u|^2 \,dx
\leq C R^{n-1} 
$$
for all $R \ge 1$.
\end{corollary}

\begin{remark}
The proof of Corollary \ref{tupper} is trivial for $1$-d solutions. Indeed, it is easy to check that $\int_{- \infty}^{+ \infty} \left\lbrace \frac{1}{2} (u_* ')^2 + \frac{1}{4}(1- u_*^2)^2 \right\rbrace dy < \infty$ and, as a consequence, by applying Fubini's theorem on a cube larger than $B_R$, that \eqref{upperbou} holds.
This argument also shows that the exponent $n-1$ in \eqref{upperbou} is optimal (since it cannot be improved for $1$-d solutions).
\end{remark}

The estimates in Corollary \ref{tupper} are fundamental in the proofs of a conjecture of De Giorgi that we treat in the next subsection.

The estimate \eqref{upperbou} was first proved by Ambrosio and the first author in \cite{ambcab}. Later on, in \cite{AAC} Alberti, Ambrosio, and the first author discovered that monotone solutions with limits are minimizers (Theorem \ref{alba} above). This allowed to simplify the original proof of the energy estimates found in \cite{ambcab}, as follows.

\begin{proof}[of Corollary \ref{tupper}]
Since $u$ is a minimizer by Theorem \ref{alba} (or by hypothesis), we can
perform a simple energy comparison argument. Indeed, let
$\phi_R\in C^\infty(\RR^{n})$ satisfy $0\le\phi_R\le 1$ in $\RR^{n}$,
$\phi_R\equiv 1$ in $B_{R-1}$,
$\phi_R\equiv 0$ in $\RR^{n}\setminus B_R$, and 
$\Vert\nabla\phi_R\Vert_\infty\leq 2$. Consider  
$$
v_R:=(1-\phi_R) u+\phi_R.
$$

Since $v_R \equiv u$ on $\pa B_R$, we can compare
the energy of $u$ in $B_R$ with that of $v_R$. 
We obtain
\begin{multline*}
\int_{B_R}\left\{\frac{1}{2} |\nabla u|^2+ G(u) \right\}dx \leq 
\int_{B_R}\left\{\frac{1}{2}|\nabla v_R|^2+ G(v_R) \right\}dx 
\\ =\int_{B_R\setminus B_{R-1}}
\left\{\frac{1}{2}|\nabla v_R|^2+ G(v_R) \right\}dx \leq 
C|B_R\setminus B_{R-1}| \le  CR^{n-1}
\end{multline*}
for every $R \ge 1$, with $C$ independent of $R$. In the second inequality of the chain above we used that $\frac{1}{2}|\nabla v_R|^2+ G(v_R) \le C$ in $B_R\setminus B_{R-1}$ for some constant $C$ independent
of~$R$. This is a consequence of the following exercise.
\qed

\end{proof}

\begin{exercise}
Prove that if $u$ is a solution of a semilinear equation $- \De u = f(u)$ in $\RR^n$ and $|u| \le 1$ in $\RR^n$, where $f$ is a continuous nonlinearity, then $|u| + | \na u| \le C$ in $\RR^n$ for some constant $C$ depending only on $n$ and $f$. See \cite{ambcab}, if necessary, for a proof.
\end{exercise}

\subsection{A conjecture of De Giorgi}

In 1978, E. De Giorgi \cite{DG} stated the following conjecture:

\smallskip\noindent
{\em 
{\rm {\bf Conjecture (DG).}}  
Let $u:\RR^{n}\to (-1,1)$ be a solution of the Allen-Cahn equation \eqref{AlCa}
satisfying the monotonicity condition \eqref{monoton}.
Then, $u$ is a $1$-d solution (or equivalently, all level sets $\{u=s\}$ of $u$ are
hyperplanes), at least if $n\leq 8$.
}

\smallskip

This conjecture was proved in 1997 for $n=2$ by Ghoussoub and Gui \cite{GG}, and in 2000 for $n=3$ by
Ambrosio and Cabr\'e \cite{ambcab}.
Next we state a deep result of Savin \cite{Sa} under the only assumption of minimality. This is the semilinear analogue of Simons Theorem \ref{thmcone} and Remark \ref{remarkchemiserve:1.17} on minimal surfaces. As we will see, Savin's result leads to a proof of Conjecture (DG) for $n \le 8$ if the additional condition \eqref{limiting} on limits is assumed.

\begin{theorem}[\textbf{Savin \cite{Sa}}]\label{teosavin}
Assume that $n\leq 7$ and that $u$ is a minimizer of \eqref{AlCa} in~$\RR^n$. 
Then, $u$ is a $1$-d solution.
\end{theorem}

The hypothesis $n\leq 7$ on its statement is sharp.
Indeed, in 2017 Liu, Wang, and Wei \cite{LWW} have shown the existence of a minimizer in $\RR^8$ whose level sets are not hyperplanes. Its zero level set is asymptotic at infinity to the Simons cone. However, a canonical solution described in Subsection \ref{subsec:saddlesha} (and whose zero level set is exactly the Simons cone) is still not known to be a minimizer in $\RR^8$.

Note that Theorem \ref{teosavin} makes no assumptions on the monotonicity or
the limits at infinity of the solution. To prove Conjecture (DG) using Savin's result
(Theorem~ \ref{teosavin}), one needs to make the further assumption \eqref{limiting} on the limits only to guarantee, by Theorem~\ref{alba}, that the solution is actually a minimizer.
Then, Theorem~\ref{teosavin} (and the gain of one more dimension, $n=8$, thanks to the monotonicity of the solution)
leads to the proof of Conjecture (DG) for monotone solutions with limits $\pm 1$.

However, for $4\leq n\leq 8$ the conjecture in its original statement (i.e., without the limits $\pm 1$ as hypothesis) is still open.
To our knowledge no clear evidence is known about its validity
or not.

The proof of Theorem \ref{teosavin} uses an improvement of flatness result for the Allen-Cahn equation developed by Savin, as well as Theorem \ref{thmcone} on the non-existence of stable minimal cones in dimension $n \le 7$.

Instead, the proofs of Conjecture (DG) in dimensions $2$ and $3$ are much simpler. They use the energy estimates of Corollary \ref{tupper} and a Liouville-type theorem developed in \cite{ambcab} (see also \cite{AAC}). As explained next, the idea of the proof originates in the paper \cite{BCN} by Berestycki, Caffarelli, and Nirenberg.

\smallskip
\noindent
{\bf Motivation for the proof of Conjecture (DG) for $n \le 3$.}
In \cite{BCN} the authors made the following heuristic observation. From the equation
$- \De u = f(u)$ and the monotonicity assumption \eqref{monoton}, by differentiating we find that 
\begin{equation}
\label{eq:remforseno}
u_{x_n}>0 \, \mbox{ and } \, L u_{x_n} := \left( - \De - f'(u) \right) u_{x_n}=0 \, \mbox{ in } \RR^n .
\end{equation} 
If we were in a bounded domain $\Om\subset \RR^n$ instead of $\RR^n$ (and we forgot about boundary conditions), from \eqref{eq:remforseno}, we would deduce that $u_{x_n}$ is the first eigenfunction of $L$ and that its first eigenvalue is $0$. As a consequence, such eigenvalue is simple. But then, since we also have that
$$
L u_{x_i} = \left( - \De - f'(u) \right) u_{x_i}=0 \quad \mbox{for } \, i=1, \dots, n-1,
$$
the simplicity of the eigenvalue would lead to 
\begin{equation}\label{eq:remforseno2}
u_{x_i}= c_i u_{x_n} \quad \mbox{for } \, i=1, \dots, n-1 ,
\end{equation}
where $c_i$ are constants. Now, we would conclude that $u$ is a 1-d solution, by the following exercise.
\begin{exercise}
Check that \eqref{eq:remforseno2}, with $c_i$ being constants, is equivalent to the fact that $u$
is a 1-d solution.
\end{exercise}
To make this argument work in the whole $\RR^n$, one needs a Liouville-type theorem. For $n=2$ it was proved in \cite{BCN} and \cite{GG}. Later, \cite{ambcab} used it to prove Conjecture (DG) in $\RR^3$ after proving the crucial energy estimate \eqref{upperbou}. The Liouville theorem requires the right hand side of \eqref{upperbou} to be bounded by $C R^2= C R^{3-1}$.

\smallskip

In 2011, del Pino, Kowalczyk, and Wei \cite{dP} established 
that Conjecture~(DG) does not hold for $n\geq 9$ 
-- as suggested in De Giorgi's original statement.

\begin{theorem}[del Pino-Kowalczyk-Wei \cite{dP}]\label{thm:delP K W}
If $n \ge 9$, there exists a solution of \eqref{AlCa}, satisfying \eqref{monoton} and \eqref{limiting}, and which is not a $1$-d solution.
\end{theorem}

The proof in \cite{dP} uses crucially the minimal graph in $\RR^9$ built by Bombieri, De Giorgi, and Giusti in \cite{BdGG}. This is a minimal surface in $\RR^9$ given by the graph of a function $\phi: \RR^8 \to \RR$ which is antisymmetric with respect to the Simons cone.
The solution of Theorem \ref{thm:delP K W} is built in such a way that its zero level set stays at finite distance from the Bombieri-De Giorgi-Giusti graph.

We consider next a similar object to the previous minimal graph, but in the context of the Allen-Cahn equation: a solution $u: (\RR^8 =) \RR^{2m} \to \RR$ which is antisymmetric with respect to the Simons cone.

\subsection{The saddle-shaped solution vanishing on the Simons cone}\label{subsec:saddlesha}

As in Section \ref{sec:mincones}, let $m \ge 1$ and denote by $\cSC$
the Simons cone \eqref{def:simonscone}.
For $x=(x_1,\dots, x_{2m})\in\RR^{2m}$, $s$ and $t$ denote the two
radial variables
\begin{equation}\label{coor}
s =   \sqrt{x_1^2+...+x_m^2} \quad \mbox{ and } \quad t  =  \sqrt{x_{m+1}^2+...+x_{2m}^2}.
\end{equation}
The Simons cone is given by 
$$
\cSC=\{s=t\}=\partial E, \quad\text{ where }
E =\{s>t\}.
$$ 


\begin{definition}[Saddle-shaped solution]\label{def:saddlesolution} 
We  say  that $u:\RR^{2m}\rightarrow\RR$ is a {\it saddle-shaped
solution}  (or simply a saddle solution) of the Allen-Cahn equation 
\begin{equation}\label{eq2m}
-\Delta u= u- u^3 \quad {\rm in }\;\RR^{2m}
\end{equation}
whenever $u$ is a
solution of
\eqref{eq2m} and, with $s$ and $t$ defined by \eqref{coor},
\renewcommand{\labelenumi}{$($\alph{enumi}$)$}
\begin{enumerate}
\item $u$ depends only on the variables $s$ and $t$. We write
$u=u(s,t)$; 
\item $u>0$ in $E:=\{s>t\}$; 
\item  $u(s,t)=-u(t,s)$ in $\RR^{2m}$.
\end{enumerate}
\end{definition}

\begin{figure}[htbp]
\centering
\includegraphics[scale=.25]{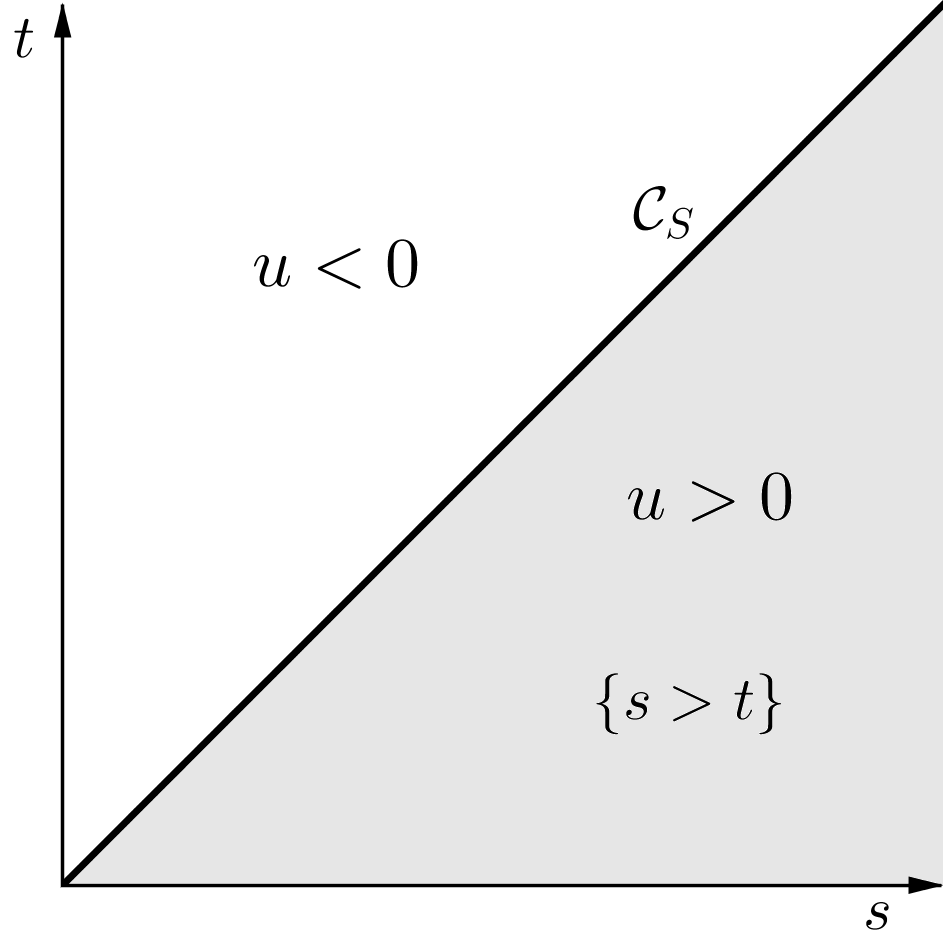}
\caption{The saddle-shaped solution $u$ and the Simons cone $\cSC$}
\label{fig:10}       
\end{figure}

\begin{remark}
Notice that if $u$ is a saddle-shaped solution, then we have $u=0$ on $\cSC$ (see Figure~\ref{fig:10}).
\end{remark}

While the existence of a saddle-shaped solution is easily established, its uniqueness is more delicate. This was accomplished in 2012 by the first author in \cite{Ca}.

\begin{theorem}[Cabr\'e \cite{Ca}] \label{unique}
For every even dimension  $2m\geq 2$,
there exists a unique saddle-shaped solution $u$ of \eqref{eq2m}.
\end{theorem}

Due to the minimality of the Simons cone when $2m \ge 8$ (and also because of the minimizer from \cite{LWW} referred to after Theorem \ref{teosavin}), the saddle-shaped solution is expected to be a minimizer when $2m \ge 8$:

\begin{open problem}
Is the saddle-shaped solution a minimizer of \eqref{eq2m} in $\RR^8$, or at least in higher even dimensions?
\end{open problem}

Nothing is known on this open problem except for the following result.
It establishes stability (a weaker property than minimality) for $2m \ge 14$. Below, we sketch its proof.

\begin{theorem}[Cabr\'e \cite{Ca}]\label{stab} 
If $2m\geq 14$, the saddle-shaped solution $u$ of \eqref{eq2m} is stable in $\RR^{2m}$, in the sense of the following definition.
\end{theorem}

\begin{definition}[Stability]
We say that a
solution $u$ of $- \De u = f(u)$ in $\RR^n$ is {\it stable} if the second variation of the energy with respect to compactly
supported perturbations $\xi$ is nonnegative. That is, if
\begin{equation*}
\int_{{\RR}^{n}}\left\{ |\nabla\xi|^2-f'(u)\xi^2\right\}dx\geq 0 \quad \mbox{ for all } \;\xi\in C^1_c(\RR^{n}) .
\end{equation*}
\end{definition}

In the rest of this section, we will take $n=2m$ and $f$ to be the Allen-Cahn nonlinearity, i.e.,
$f (u) = u- u^3 .$

\begin{sketch}[of Theorem \ref{stab}]
Notice that
\begin{equation}\label{eqst}
u_{ss}+u_{tt}+(m-1){\Big (}\frac{u_s}{s}+\frac{u_t}{t}{\Big
)}+ f(u) =0 ,
\end{equation}
for $s>0$ and $t>0$,
is equation \eqref{eq2m} expressed in the 
$(s,t)$ variables.
Let us introduce the function
\begin{equation}\label{0001:eq:deffi}
\varphi := t^{-b}u_s -s^{-b}u_t .
\end{equation}
Differentiating \eqref{eqst} with respect to $s$ (and to $t$), one finds equations satisfied by $u_s$ (and by $u_t$) -- and which involve a zero order term with coefficient $f'(u)$. These equations, together with some more delicate monotonicity properties of the saddle-shaped solution established in \cite{Ca}, can be used to prove the following fact.

For $2m\geq 14$, one can choose $b>0$ in \eqref{0001:eq:deffi} (see \cite{Ca} for more details) such that $\varphi$ is a positive supersolution of the linearized problem, i.e.:
\begin{equation}\label{posivar}
\varphi >0 \quad \text{in } \{st> 0\},
\end{equation}
\begin{equation}\label{superpf}
\{\Delta + f'(u)\} \varphi \leq 0 \quad \text{in } \RR^{2m}\setminus\{st=0\}=\{st>0\} .
\end{equation}

Next, using \eqref{posivar} and \eqref{superpf}, we can verify the stability
condition of $u$ for any $C^1$ test function
$\xi = \xi(x)$ with compact support in $\{st>0\}$. Indeed, multiply \eqref{superpf}
by $\xi^2/\varphi$ and integrate by parts to get
\begin{eqnarray*}
\int_{\{st>0\}} f'(u)\,\xi^2 \, dx  
& = &\int_{\{st>0\}} f'(u) \varphi\, \frac{\xi^2}{\varphi} \, dx \\
& \leq & \int_{\{st>0\}} -\Delta \varphi \, \frac{\xi^2}{\varphi}  \, dx \\
& = &\int_{\{st>0\}} 
\nabla \varphi \, \nabla\xi\, \frac{2\xi}{\varphi}  \, dx 
-\int_{\{st>0\}} \frac{|\nabla \varphi|^2}{\varphi^2}\, \xi^2 \, dx .
\end{eqnarray*}
Now, using the Cauchy-Schwarz inequality, we are led to
\begin{equation*}
\int_{\{st>0\}} f'(u)\,\xi^2 \, dx \leq \int_{\{st>0\}} |\nabla\xi|^2 \, dx .
\end{equation*}

Finally, by a cut-off argument we can prove that this same inequality holds also for every function $\xi \in C^1_c(\RR^{2m})$.
\qed

\end{sketch}

\begin{remark}
Alternatively to the variational proof seen above, another way to establish stability from the existence of a positive supersolution to the linearized problem is by using the maximum principle (see \cite{BNV} for more details).
\end{remark}

\section{Blow-up problems} 

In this final section, we consider positive solutions of the semilinear problem

\begin{equation}\label{pb}
\left\{
\begin{array}{rcll}
-\De u &=& f(u) & \quad\mbox{in } \Om \\
u &>& 0  & \quad\mbox{in } \Om \\
u &=& 0  & \quad\mbox{on } \pa\Om ,\\
\end{array}\right.
\end{equation}
where $\Om \subset \RR^n$ is a smooth bounded domain, $n\geq 1$, and $f: \RR^+ \to \RR$ is $C^1$.

The associated energy functional is
\begin{equation}\label{eq:energiaultimasec}
E_\Om (u):=\int_\Omega \left\{ \frac{1}{2} |\nabla u|^2 - F(u) \right\} dx ,
\end{equation}
where $F$ is such that $F' = f$.

\subsection{Stable and extremal solutions. A singular stable solution for $n \ge 10$}

We define next the class of stable solutions to \eqref{pb}. It includes any local minimizer, i.e., any minimizer of \eqref{eq:energiaultimasec} under small perturbations vanishing on $\pa \Om$.

\begin{definition}[Stability]
A solution $u$ of \eqref{pb} is said to be \emph{stable} if 
the second variation of the energy with respect to $C^1$ perturbations $\xi$ vanishing on $\pa \Om$ is nonnegative. That is, if
\begin{equation} \label{stability}
\int_{\Omega} f'(u)\xi^2\,dx
\leq
\int_{\Omega} |\nabla\xi|^2\,dx \quad\textrm{for all } \xi \in C^1 (\ol{\Om} ) \textrm{ with } \xi_{|\pa \Om} \equiv 0.
\end{equation}
\end{definition}


There are many nonlinearities for which \eqref{pb} admits a (positive) stable solution. 
Indeed, replace $f(u)$ by $\lambda f(u)$
in \eqref{pb}, with $\lambda\geq 0$:
\begin{equation}\label{pbla}
\left\{
\begin{array}{rcll}
-\Delta u &=& \la f(u) & \quad\mbox{in $\Omega$}\\
u &=& 0  & \quad\mbox{on $\pa\Omega$.}\\
\end{array}\right.
\end{equation}
Assume that $f$ is positive, nondecreasing, and superlinear at $+\infty$, that is,
\begin{equation} \label{hypf}
f(0) > 0, \quad f'\geq 0 \quad\textrm{and}\quad\lim_{t\to +\infty}\frac{f(t)}t=+\infty.
\end{equation}
Note that also in this case we look for positive solutions (when $\la>0$), since $f>0$.
We point out that, for $\la>0$, $u\equiv 0$ is not a solution.

\begin{proposition}\label{prop:extremal}
Assuming \eqref{hypf}, there exists an extremal parameter $\lambda^*\in (0,+\infty)$ such that if $0\le\lambda<\lambda^*$ then
\eqref{pbla} admits a minimal stable classical solution $u_\lambda$. Here ``minimal'' means the smallest among all the solutions, while ``classical'' means of class $C^2$. Being classical is a consequence of $u_\la \in L^{\infty} ( \Om)$ if $\la < \la^*$.

On the other hand, if $\lambda>\lambda^*$ then \eqref{pbla} has no classical solution.

The family of classical solutions  $\left\{ u_{\lambda }:0\le \lambda <\lambda ^{*}\right\}$ 
is increasing in $\lambda$, and its limit as $\lambda\uparrow\lambda^{*}$ is a weak solution $u^*=u_{\lambda^*}$ of \eqref{pbla} for $\la=\la^*$. 
\end{proposition}

\begin{definition}[Extremal solution]
The function $u^*$ given by Proposition \ref{prop:extremal} is called \emph{the extremal solution} of \eqref{pbla}.
\end{definition}

For a proof of Proposition \ref{prop:extremal} see the book \cite{Dupaigne} by L. Dupaigne.
The definition of weak solution (the sense in which $u^*$ is a solution) requires $u^* \in L^1 (\Om)$, $f(u^* ) \mathrm{dist}(\cdot, \pa \Om) \in L^1 (\Om)$, and the equation to be satisfied in the distributional sense after multiplying it by test functions vanishing on $\pa \Om$ and integrating by parts twice (see \cite{Dupaigne}). Other useful references regarding extremal and stable solutions are \cite{Brezis}, \cite{Cextremal}, and \cite{CabBound}.

Since 1996, Brezis has raised several questions regarding stable and extremal solutions; see for instance \cite{Brezis}. They have led to interesting works, some of them described next. One of his questions is the following.

\medskip
\noindent
\textbf{Question (Brezis).} Depending on the dimension $n$ or on the domain $\Om$, is the extremal solution $u^*$ of \eqref{pbla} bounded (and therefore classical) or is it unbounded?
More generally, one may ask the same question for the larger class of stable solutions to \eqref{pb}.

\medskip

\begin{figure}[htbp]
\centering
\includegraphics[scale=.25]{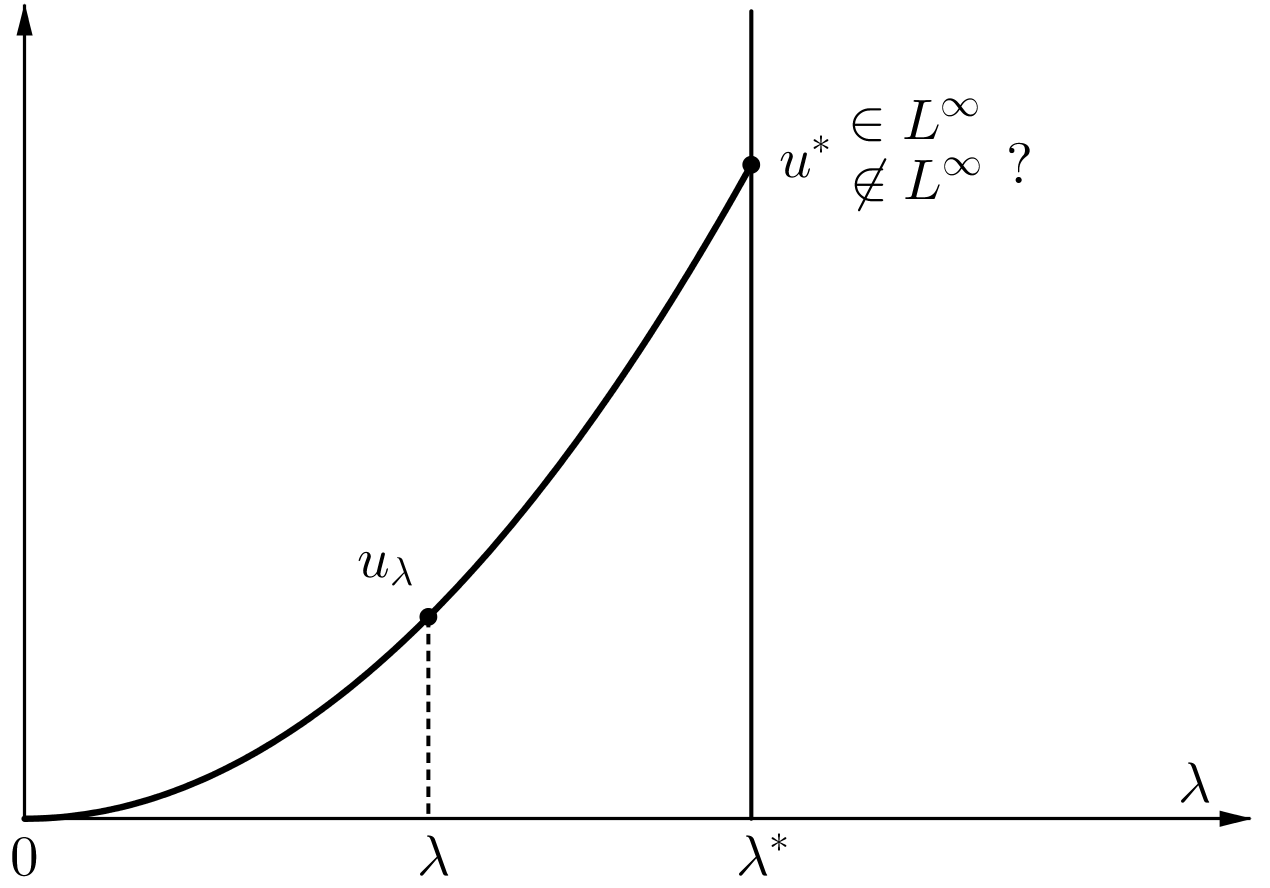}
\caption{The family of stable solutions $u_\la$ and the extremal solution $u^*$}
\label{fig:11}       
\end{figure}

The following is an explicit example of stable unbounded (or singular) solution.

It is easy to check that, for $n\geq 3$, the function $\tilde u=-2\log\vert{x}\vert$ is a solution of \eqref{pb} in $\Om =B_1$, 
the unit ball, for $f(u)= 2(n-2) e^u$. 
%
Let us now consider the linearized operator at $\tilde u$, which is given by  
$$
-\Delta -2(n-2)e^{\tilde u} =
-\Delta  -\frac{2(n-2)}{\vert{x}\vert^2}.
$$
If $n\ge 10$, then its first Dirichlet eigenvalue in $B_1$ is nonnegative.
This is a consequence of {\it Hardy's inequality} \eqref{Hardyin}:
$$
\frac{(n-2)^2}{4}\int_{B_1}\frac{\xi^2}{\vert{x}\vert^2} \, dx \;\; \leq\;\;
\int_{B_1} \vert\nabla \xi\vert^2 dx
\qquad\quad
\mbox{ for every } \xi\in H^1_0(B_1),
$$
and the fact that $2(n-2)\leq (n-2)^2/4$ if $n\geq 10$.
Thus we proved the following result.
\begin{proposition}\label{prop:dimensopt}
For $n\geq 10$, $\tilde u = -2\log\vert{x}\vert$ is an $H^1_0(B_1)$ stable weak solution of $-\De u = 2 (n-2) e^u $ in $B_1$, $u>0$ in $B_1$, $u=0$ on $\pa B_1$.
\end{proposition}

Thus, in dimensions $n\geq 10$ there exist unbounded $H^1_0$ stable weak solutions of \eqref{pb}, 
even in the unit ball and for the exponential nonlinearity.
It is believed that $n \ge 10$ could be the optimal dimension for this fact, as we describe next.

\subsection{Regularity of stable solutions. The Allard and Michael-Simon Sobolev inequality}

The following results give $L^{\infty}$ bounds for stable solutions. To avoid technicalities we state the bounds for the extremal solution but, more generally, they also apply to every stable weak solution of \eqref{pb} which is the pointwise limit of a sequence of bounded stable solutions to similar equations (see \cite{Dupaigne}). 


\begin{theorem}[Crandall-Rabinowitz \cite{CR}]
\label{thm:CranRab}
Let $u^*$ be the extremal solution of \eqref{pbla} with $f(u)=e^u$ or $f(u)=(1+u)^p$, $p>1$. If $n\le 9$, then $u^* \in L^\infty (\Om)$.
\end{theorem}

\begin{sketch}[in the case $f(u)= e^u$]
Use the equation in \eqref{pbla} for the classical solutions $u=u_\la$ ($\la < \la^*$), together with the stability condition \eqref{stability} for the test function $\xi=e^{\al u} -1$ (for a positive exponent $\al$ to be chosen later).
More precisely, start from \eqref{stability} -- with $f'$ replaced by $\la f'$ -- and to proceed with $\int_\Om \al^2 e^{2 \al u} |\na u|^2$, write $ \al^2 e^{2 \al u} | \na u|^2 = ( \al / 2) \na \left( e^{2 \al u} -1 \right) \na u$, and integrate by parts to use \eqref{pbla}.
For every $\al < 2$, verify that this leads, after letting $\la \uparrow \la^*$,
to $e^{u^*} \in L^{2 \al +1} (\Om)$. As a consequence, by Calder\'{o}n-Zygmund theory and Sobolev embeddings, $u^* \in W^{2, 2 \al +1} (\Om) \subset L^\infty (\Om)$ if $2 ( 2 \al +1) > n$. This requires that $n \le 9$.
\qed
\end{sketch}

Notice that the nonlinearities $f(u)=e^u$ or $f(u)=(1+u)^p$ with $p>1$ satisfy~\eqref{hypf}.

In the radial case $\Om=B_1$ we have the following result.

\begin{theorem}[Cabr\'e-Capella \cite{CC}]
\label{corolrad}
Let $u^*$ be the extremal solution of \eqref{pbla}. Assume that $f$ satisfies \eqref{hypf} and that $\Om=B_1$.
If $1\leq n\leq 9$, then $u^* \in L^\infty (B_1)$. 
\end{theorem}

As mentioned before, this theorem also holds for every $H_0^1 (B_1)$ stable weak solution of \eqref{pb}, for any $f \in C^1$.
Thus, in view of Proposition \ref{prop:dimensopt}, the dimension $n \le 9$ is optimal in this result.

We turn now to the nonradial case and we present the currently known results. First, in 2000 Nedev solved the case $n \le 3$.

\begin{theorem}[Nedev \cite{ND}]\label{thm3pre}
Let $f$ be convex and satisfy \eqref{hypf}, and $\Omega\subset\RR^{n}$ be a smooth bounded domain.
If $n\leq 3$, then $u^* \in L^\infty(\Omega)$.
\end{theorem}

In 2010, Nedev's result was improved to dimension four:

\begin{theorem}[Cabr\'e \cite{Cabre}; Villegas \cite{V}]\label{thm3}
Let $f$ satisfy \eqref{hypf}, $\Omega\subset\RR^{n}$ be a smooth bounded domain, and $1 \le n\leq 4$.
If $n\in \{3,4\}$ assume either that $f$ is a convex nonlinearity or that $\Omega$ is a convex domain. 
Then, $u^* \in L^\infty(\Omega)$.
\end{theorem}

For $3\leq n \leq 4$, \cite{Cabre} requires $\Omega$ to be convex, while $f$ needs not be
convex. Some years later, S. Villegas~\cite{V} succeeded to use both \cite{Cabre} and \cite{ND} when $n=4$
to remove the requirement that $\Omega$ is convex by further assuming that $f$ is convex.

\begin{open problem}
For every $\Om$ and for every $f$ satisfying \eqref{hypf}, is the extremal solution $u^*$ -- or, in general, $H_0^1$ stable weak solutions of \eqref{pb} -- always bounded in dimensions $5,6,7,8,9$?
\end{open problem}

We recall that the answer to this question is affirmative when $\Om = B_1$, by Theorem \ref{corolrad}. We next sketch the proof of this radial result, as well as the regularity theorem in the nonradial case up to $n \le 4$.
In the case $n=4$, we will need the following remarkable result.

\begin{theorem}[Allard; Michael and Simon]
\label{Sobolev}
Let  $M\subset  \RR^{m+1}$ be an immersed smooth $m$-dimensional 
compact hypersurface without boundary.

Then, for every $p\in  [1, m)$,
there exists a constant $C = C(m,p)$ depending only on the dimension $m$ 
and the exponent $p$ such that, for every $C^\infty$ function $v : M  \to \RR$, 
\begin{equation}
\label{MSsob}
\left( \int_M |v|^{p^*} \,dV \right)^{1/p^*} \leq C(m,p)
\left( \int_ M (|\nabla v|^p +  | \cH v|^p) \,dV \right)^{1/p},
\end{equation}
where $\cH$ is the mean curvature of $M$ and $p^* = mp/(m - p)$.
\end{theorem}

This theorem dates from 1972 and has its origin in an important result of Miranda from 1967. It stated that \eqref{MSsob} holds with $\cH=0$ if $M$ is a minimal surface in $\RR^{m+1}$.
See the book \cite{Dupaigne} for a proof of Theorem \ref{Sobolev}.

\begin{remark}
Note that this Sobolev inequality contains a term involving the mean curvature of $M$ on its right-hand side. This fact makes, in a remarkable way, that the constant $C(m,p)$ in the inequality does not depend on the geometry of the manifold~$M$.
\end{remark}

\begin{sketch}[of Theorems \ref{corolrad}, \ref{thm3pre}, and \ref{thm3}] 
For Theorem \ref{thm3pre} the test function to be used is $\xi=h(u)$, for some $h$ depending on $f$ (as in the proof of Theorem~\ref{thm:CranRab}).

Instead, for Theorems \ref{corolrad} and \ref{thm3}, the proofs start by writing 
the stability condition \eqref{stability} for the test 
function $\xi= \tilde{c} \eta$, where $\eta|_{\partial\Omega}\equiv0$.
This was motivated by the analogous computation that we have presented for minimal surfaces right after Remark \ref{rem:Jacobi operator}.
Integrating by parts, one easily deduces that
\begin{equation}\label{stabc}
\int _{\Omega} \left( \Delta \tilde{c} +f'(u) \tilde{c} \right) \tilde{c} \eta ^{2}\,dx \le
\int _{\Omega} \tilde{c}^{2}\left|\nabla \eta\right|^{2}\,dx.
\end{equation}
Next, a key point is to choose a function $\tilde{c}$ satisfying an appropriate equation for
the linearized operator $\Delta + f'(u)$. In the radial case (Theorem \ref{corolrad}) the choice of $\tilde{c}$ and the final choice of $\xi$ are
\begin{equation*}
\tilde{c} =u_{r} \quad\text{ and }\quad\xi=u_r r(r^{-\alpha}-(1/2)^{-\alpha})_{+} ,
\end{equation*}
where $r=|x|$, $\alpha>0$, and $\xi$ is later truncated near the origin to make it Lipschitz.
The proof in the radial case is quite simple after computing the equation satisfied by $u_r$.


For the estimate up to dimension 4 in the nonradial case (Theorem~\ref{thm3}), \cite{Cabre} takes
\begin{equation}\label{choice4}
\tilde{c} =\left|\nabla u\right| \quad\text{ and }\quad\xi=\left|\nabla u\right|\varphi(u) ,
\end{equation}
where, in dimension $n=4$, $\varphi$ is chosen depending on the solution $u$ itself.

We make the choice \eqref{choice4} and, in particular, we take
$\tilde{c} =\left|\nabla u\right|$ in \eqref{stabc}.
It is easy to check that, in the set 
$\left\{\left|\nabla u\right|>0\right\}$, we have
\begin{equation}\label{dipassaggio}
\left(\Delta + f'(u)\right)|\nabla u|=\frac{1}{|\nabla u|}
\left(\sum_{i,j}u_{ij}^2-\sum_i\left(\sum_{j}u_{ij}\frac{u_j}
{|\nabla u|}\right)^2\right).
\end{equation}
Taking an orthonormal basis in which the last vector is the normal $\nabla u/|\nabla u|$ to the level set of $u$ (through a given point $x \in \Om$), and the other vectors are the principal directions of the level set at $x$, one easily sees that \eqref{dipassaggio} can be written as
\begin{equation}\label{eq:grad}
\left( \Delta+f'(u)\right)\left|\nabla u\right|=
\frac{1}{\left|\nabla u\right|} \left(\left|\nabla_T\left|\nabla u\right|\right|^2+
\left|A\right|^2\left|\nabla u\right|^2\right)\quad \text{in}\ 
\Omega\cap\left\{\left|\nabla u\right|>0\right\},
\end{equation}
where $\left|A\right|^2=\left|A\left(x\right)\right|^2$ is the squared norm of 
the second fundamental form of the level set of $u$ passing through a 
given point $x\in\Omega\cap\left\{\left|\nabla u\right|>0\right\}$, i.e., the 
sum of the squares of the principal curvatures of the level set. In the notation of the first section on minimal surfaces, $|A|^2 = c^2$.
On the other hand, as in that section $\nabla_T = \de$ denotes the tangential gradient to the level set. 
Thus, \eqref{eq:grad} involves geometric information of the level sets 
of $u$.

Therefore, using the stability condition \eqref{stabc},
we conclude that
\begin{equation}\label{semi1}
\int_{\left\{\left|\nabla u\right|>0\right\}} 
\left( |\nabla_T |\nabla u||^2 +|A|^2|\nabla u|^2\right)\eta^2\,dx
\leq \int_\Omega |\nabla u|^2 |\nabla \eta|^2 \,dx .
\end{equation}

Let us define
$$
T:=\max_{\ol{\Omega} }u=\left\|u\right\|_{L^\infty (\Omega)}
\quad\text{ and }\quad \Gamma_s:=\left\{x\in\Omega:u(x)=s\right\}
$$
for $s\in(0,T)$.  

We now use \eqref{semi1} with  
$\eta=\varphi(u)$, 
where $\varphi$ is a Lipschitz function in $\left[0,T\right]$ 
with $\varphi(0)=0$.
The right hand side of \eqref{semi1} becomes
\begin{eqnarray*}
\int_{\Omega}\left|\nabla u\right|^2\left|\nabla\eta\right|^2dx&=&
\int_{\Omega}\left|\nabla u\right|^4\varphi'(u)^2dx\\
&=&\int_{0}^{T}\left(\int_{\Gamma_s}\left|\nabla u\right|^3\,dV_s
\right)\varphi'(s)^2\,ds,
\end{eqnarray*}
by the {\it coarea formula}. Thus, \eqref{semi1} can be written as
\begin{eqnarray*}
&&\hspace{-1cm}\int_{0}^{T}\left(\int_{\Gamma_s}\left|\nabla u\right|^3\,dV_s
\right)\varphi'(s)^2\,ds\\
&&\geq\int_{\left\{\left|\nabla 
u\right|>0\right\}}\left(\left|\nabla_T\left|
\nabla u\right|\right|^2+\left|A\right|^2\left|\nabla 
u\right|^2\right)\varphi(u)^2dx\\
&&=\int_{0}^{T}\left(\int_{\Gamma_s\cap\left\{\left|\nabla 
u\right|>0\right\}}\frac{1}{\left|\nabla u\right|}
\left(\left|\nabla_T\left|
\nabla u\right|\right|^2+\left|A\right|^2\left|\nabla 
u\right|^2\right)\,dV_s\right)\varphi(s)^2\,ds\\
&&=\int_{0}^{T}\left(\int_{\Gamma_s\cap\left\{\left|\nabla 
u\right|>0\right\}}
\left( 4\left|\nabla_T\left|
\nabla u\right|^{1/2}\right|^2+\left(\left|A\right|\left|\nabla 
u\right|^{1/2}\right)^{2}\right) \,dV_s\right)\varphi(s)^2\,ds .
\end{eqnarray*}
We conclude that
\begin{equation}
\int_0^T h_1(s) \varphi(s)^2 \,ds
\leq
\int_0^T h_2(s)  \varphi'(s)^2 \,ds,
\label{semi3}
\end{equation}
for all Lipschitz functions 
$\varphi:\left[0,T\right]\rightarrow\mathbb{R}$ with 
$\varphi(0)=0$, where 
\begin{equation*}
h_1(s):=\int_{\Gamma_s} 
\left( 4|\nabla_T |\nabla u|^{1/2}|^2 +\left( |A||\nabla u|^{1/2} \right)^2\right) \,dV_s\, ,\quad
h_2(s):=\int_{\Gamma_s} |\nabla u|^3\,dV_s 
\end{equation*}
for every regular value $s$ of $u$.
We recall that, by Sard's theorem, almost every $s\in(0,T)$ is a regular value of $u$.

Inequality \eqref{semi3}, with $h_1$ and $h_2$ as defined above, leads to a bound for $T$
(that is, to an $L^{\infty}$ estimate and hence to Theorem \ref{thm3}) after choosing an 
appropriate test function $\varphi$ in \eqref{semi3}. 
In dimensions 2 and 3 we can choose a simple function $\varphi$ in \eqref{semi3} and use well known geometric inequalities 
about the curvature of manifolds (note that $h_1$ involves the curvature 
of the level sets of $u$). Instead, in dimension 4 we need to use the geometric Sobolev inequality of Theorem \ref{Sobolev} on each level set of $u$.
Note that $\cH^2 \le (n-1) |A|^2$.
This gives the following lower bound for $h_1 (s)$:
$$
c(n) \left( \int_{\Ga_s} | \na u |^{\frac{n-1}{n-3} }  \right)^{\frac{n-3}{n-1}} \le h_1 (s).
$$
Comparing this with $h_2 (s)$, which appears in the right hand side of \eqref{semi3}, we only know how to derive an $L^{\infty}$-estimate for $u$ (i.e., a bound on $T= \max u$) when the exponent $(n-1) / (n-3)$ in the above inequality is larger than or equal to the exponent $3$ in $h_2 (s)$.
This requires $n \le 4$.
See \cite{Cabre} for details on how the proof is finished.
\qed
\end{sketch}


\section{Appendix: a calibration giving the optimal isoperimetric inequality}
\addcontentsline{toc}{section}{Appendix}

Our first proof of Theorem \ref{thm:SimCone} used a calibration. To understand better the concept and use of ``calibrations'', we present here another one. It leads to a proof of the isoperimetric problem.

The isoperimetric problems asks which sets in $\RR^n$ minimize perimeter for a given volume. Making the first variation of perimeter (as in Section \ref{sec:mincones}), but now with a volume constraint, one discovers that a minimizer $\Om$ should satisfy $\cH=c$ (with $c$ a constant), at least in a weak sense, where $\cH$ is the mean curvature of $\pa \Om$. Obviously, balls satisfy this equation -- they have constant mean curvature.
The isoperimetric inequality states that the unique minimizers are, indeed, balls. In other words, we have:
\begin{theorem}[The isoperimetric inequality]
We have
\begin{equation}
\label{eq:isoperimetric}
\frac{|\pa \Om|}{| \Om|^{\frac{n-1}{n}}} \ge \frac{|\pa B_1|}{| B_1|^{\frac{n-1}{n}}} 
\end{equation}
for every bounded smooth domain $\Om \subset \RR^n$.
In addition, if equality holds in \eqref{eq:isoperimetric}, then $\Om$ must be a ball.
\end{theorem}

In 1996 the first author found the following proof of the isoperimetric problem. It uses a calibration (for more details see \cite{Caisoperimetric}).

\begin{sketch}[of the isoperimetric inequality]
The initial idea was to characterize the perimeter $|\pa \Om|$ as in \eqref{calideph}-\eqref{eq:dimcalibration1}, that is, as
$$
| \pa \Om | = \sup_{ \nr X \nr_{L^\infty} \le 1} \int_{ \pa \Om} X \cdot \nu \, dH_{n-1} .
$$
Taking $X$ to be a gradient, we have that
$$
| \pa \Om | = \int_{ \pa \Om} \na u \cdot \nu \, dH_{n-1} = \int_{\pa \Om} u_\nu \, dH_{n-1} ,
$$
for every function $u$ such that $u_\nu=1$ on $\pa \Om$.
Let us take $u$ to be the solution of
\begin{equation}\label{eqappisop}
\left\{
\begin{array}{rcll}
\De u &=& c & \quad\mbox{in } \Om \\
u_\nu &=& 1  & \quad\mbox{on } \pa\Om ,\\
\end{array}\right.
\end{equation}
where $c$ is a constant that, by the divergence theorem, is given by
$$
c= \frac{|\pa \Om|}{| \Om| }.
$$
It is known that there exists a unique solution $u$ to \eqref{eqappisop} (up to an additive constant).

Now let us see that $X=\na u$ (where $X$ was the notation that we used in the proof of Theorem \ref{thm:SimCone}) can play the role of a calibration.
In fact, in analogy with Definition~\ref{def:calibrationprima} we have:

\begin{enumerate}[]
\item(i-bis) \,$\dv \na u = \frac{|\pa \Om|}{| \Om| } \, \mbox{ in } \, \Om $;
\item(ii-bis)  \,$\na u \cdot \nu = 1 \, \mbox{ on } \, \pa \Om$;
\item(iii-bis)  \,$B_1 (0) \subset \na u (\Ga_u)$, where
$$\Ga_u = \left\{ x \in \Om : u(y) \ge u(x) + \na u(x) \cdot (y-x)  \, \mbox{ for every } \, y \in \ol{\Om}  \right\} $$
is the {\it lower contact set of $u$}, that is, the set of the points of $\Om$ at which the tangent plane to $u$ stays below $u$ in $\Om$. 
\end{enumerate}
The relations (i-bis) and (ii-bis) follow immediately from \eqref{eqappisop}. 
In the following exercise, we ask to establish (iii-bis) and finish the proof of \eqref{eq:isoperimetric}.

We point out that this proof also gives that $\Om$ must be a ball if equality holds in \eqref{eq:isoperimetric}.
\qed
\end{sketch}

\begin{exercise}
Establish (iii-bis) above. For this, use a foliation-contact argument (as in the alternative proof of Theorem \ref{thm:SimCone} and in the proof of Theorem \ref{alba}), foliating now $\RR^n \times \RR$ by parallel hyperplanes.

Next, finish the proof of \eqref{eq:isoperimetric}. For this, consider the measures of the two sets in (iii-bis), compute $| \na u (\Ga_u)|$ using the {\it area formula}, and control $\det D^2 u$ using the geometric-arithmetic means inequality.

\end{exercise}

%
%
%
%

%
\begin{acknowledgement}
The authors wish to thank Lorenzo Cavallina for producing the figures of this work.

The first author is member of the Barcelona Graduate School of Mathematics and is supported by MINECO grants MTM2014-52402-C3-1-P and MTM2017-84214-C2-1-P. He is also part of the Catalan research group 2017 SGR 1392.

The second author was partially supported by PhD funds of the Universit\`{a} di Firenze and he is a member of the Gruppo Nazionale Analisi Matematica Probabilit\'{a} e Applicazioni (GNAMPA) of the Istituto Nazionale di Alta Matematica (INdAM). This work was partially written while the second author was visiting the Departament de Matem\`{a}tiques of the Universitat Polit\`ecnica de Catalunya, that he wishes to thank for hospitality and support.
\end{acknowledgement}

\end{document}